\documentclass[UTF-8,final]{siamltex}
\usepackage{amsmath}
\usepackage{epsfig}
\usepackage{graphicx}
\usepackage{amssymb}
\usepackage{CJK}
\usepackage{color}
\usepackage{subfigure} 
\usepackage{float}
\usepackage{hyperref}
\usepackage[noadjust]{cite}

\numberwithin{equation}{section}
\newtheorem{algorithm}{Algorithm}[section]
\newtheorem{remark}{Remark}[section]

\newtheorem{example}{Test}[section]

\renewcommand\abstractname{Abstract}

\renewcommand\figurename{Figure}
\renewcommand\tablename{Table}

\renewcommand\refname{References}

\allowdisplaybreaks[4]

\normalsize

\def\Ome{{\Omega}}

\def\nab{{\nabla}}

\def\p{{\partial}}

\def\cT{{\mathcal T}}

\def\div{{\mbox{\rm div\,}}}

\def\p{{\partial}}

\def\nab{\nabla}
\def\Ome{\Omega}

\def\bC{\mathbf{C}}
\def\bbf{\mathbf{f}}
\def\bu{\mathbf{u}}

\def\bv{\mathbf{v}}
\def\bw{\mathbf{w}}

\def\bg{\mathbf{g}}
\def\bn{\mathbf{n}}
\def\bH{\mathbf{H}}
\def\bV{\mathbf{V}}

\def\bP{\mathbf{P}}

\def\bV{\mathbf{V}}

\def\R{\mathbb{R}}

\begin{document}
	
	\title{ Optimal $L^2$-norm error estimate  of multiphysics finite element method for poroelasticity model and simulating brain edema\footnote{Last update: \today}}
	
	\author{
		Zhihao Ge\thanks{School of Mathematics and Statistics, Henan University, Kaifeng 475004, PR China ({\tt Email:zhihaoge@henu.edu.cn}).
			The work of this author was supported by Natural Science Foundation of Henan Province(No. 242300421047) and the National Natural Science Foundation of China(No. 12371393).}
		\and
		Yanan He\thanks{School of Mathematics and Statistics, Henan University, Kaifeng 475004, PR China ({\tt Email:hyn639@163.com}).}
		\and
		Yajie Yang\thanks{School of Mathematics and Statistics, Henan University, Kaifeng 475004, PR China ({\tt Email:819387116@qq.com}).}
	}
	
	\maketitle
	
	
	\setcounter{page}{1}
	
	
	
	\begin{abstract}
			In this paper, we derive an optimal  $L^2$-norm error estimate of the  multiphysics finite element method for the poroelasticity model  by introducing an auxiliary problem. We show some numerical tests to verify the theoretical result and apply the multiphysics finite element method to simulate the brain edema which caused by abnormal accumulation of cerebrospinal fluid in injured areas. And we investigate the effects of the key physical parameters on brain edema and observed that the permeability $K$ has the biggest influence on intracranial pressure and tissue deformation, Young's modulus $E$ and Poisson ratio $\nu$ have little effect on the maximum value of intracranial pressure, but have great effect on the tissue deformation and the developing speed of brain edema.
	\end{abstract}

	\begin{keywords}
	 Poroelasticity model, Multiphysics finite element method, $L^2$-norm error estimate, Brain edema.
	\end{keywords}
	
	\begin{AMS}
		65M12, 
		65M15, 
		65N30. 
	\end{AMS}
	
	\pagestyle{myheadings}
	\thispagestyle{plain}
	\markboth{ZHIHAO GE, YANAN HE, YAJIE YANG}{ OPTIMAL $L^2$-NORM ERROR ESTIMATE OF MFEM FOR POROELASTICITY}
	

\section{Introduction}\label{sec-1}

The theory of poroelasticity investigates the coupling process between the deformation of porous material and the internal fluid flow over time.
Poroelasticity model has a wide range of applications, such as logging technologies and borehole instability in reservoir engineering\cite{B1,B2,B3}, tunnel soil instability in geotechnics \cite{B4,B5,B6}, and $\rm CO_{2}$ sequestration in environmental engineering \cite{B7,B8}. 
In recent years, poroelasticity has been widely progressively applied to bioengineering, such as simulating biological soft tissues, including arterial walls, skin, cardiac muscle, articular cartilage and brain edema \cite{B9,B10,B11,B12,B13}.
In this paper, we consider the following 
quasi-static poroelasticity model \cite{B14,B15}:
\begin{align} \label{1.1}
	-\operatorname{div} \sigma(\bu) + \alpha \nab p = \bbf
	&\qquad \mbox{in } \Ome_T:=\Ome\times (0,T)\subset \mathbf{\R}^d\times (0,T),\\
	(c_0p+\alpha \operatorname{div} \bu)_t + \operatorname{div} \bv_f =\phi &\qquad \mbox{in } \Ome_T,
	\label{1.2}
\end{align}
where
\begin{align}\label{1.3}
	\sigma(\mathbf{u})=\mu\varepsilon(\mathbf{u})+\lambda \mathrm{tr}(\varepsilon(\mathbf{u}))\mathbf{I},\quad\varepsilon(\mathbf{u})=\frac12(\nabla\mathbf{u}+\nabla^T\mathbf{u}), \quad
	\mathbf{v}_{f}=-\frac{K}{\mu_{f}}\left(\nabla p-\rho_{f}\mathbf{g}\right).
\end{align}

To close the system  (\ref{1.1})-(\ref{1.2}), we set the following boundary and initial conditions:
\begin{align} 
\sigma(\bu)\bn-\alpha p \bn = \bbf_1
	&\qquad \mbox{on } \p\Ome_T:=\p\Ome\times (0,T),\label{1.4}\\
	\bv_f\cdot\bn= -\frac{K}{\mu_f} \bigl(\nab p -\rho_f \bg \bigr)\cdot \bn
	=\phi_1 &\qquad \mbox{on } \p\Ome_T, \label{1.5} \\
	\bu=\bu_0,\qquad p=p_0 &\qquad \mbox{in } \Ome\times\{t=0\}. \label{1.6}
\end{align}

We remark that Eq. (\ref{1.1}) is the momentum balance equations for the displacement of the medium and Eq. (\ref{1.2}) is the mass balance equation for the pressure distribution.
Here $\Omega\subset \mathbb{R} ^d\left ( d= 2, 3\right ) $ denotes a bounded domain with the boundary $\partial\Omega$. $\mathbf{u}$ denotes the displacement vector of solid and $p$ denotes the pressure of fluid. $\mathbf{f}$ is body force and $\phi$ is source term. $\sigma(\mathbf{u})$ is called the (effective) stress tensor and $\varepsilon(\mathbf{u})$ is known as the deformed Green strain tensor. $\mathbf{I}$ denotes the $d\times d$ identity matrix and $\mathbf{v}_f$ is called Darcy's law. 
Assume that the permeability tensor $K=K(x)$ is symmetric and uniformly positive definite, that is,  there exists positive constants $K_1$ and $K_2$ such that $K_1|\zeta|^2\leq K(x)\zeta\cdot\zeta\leq K_2|\zeta|^2$ for a.e. $x\in\Omega$ and $\zeta\in\mathbb{R}^d$. 
The positive Lam$\acute{e}$ constants $\lambda$ and $\mu$ are the dilation and shear modulus of elasticity, respectively. The coefficient $\alpha>0$ is the Biot-Willis constant, used to account for the pressure-deformation coupling and to measure the fluid volume forced out of the solid skeleton by dilation. The constrained specific storage coefficient $c_0\geq0$ is the combined porosity of the medium and compressibility of the fluid and solid.
$\mu_f$ denotes the fluid viscosity and  $\mathbf{g}$ denotes the acceleration of gravity. We assume $\rho_{f} \not\equiv 0$. The source terms of $\mathbf{f}_1$ and $\phi_1$ are given functions. In addition, we note that in some engineering literature the Lam$\acute{e}$ constant $\mu$ is also called the shear modulus(denoted by $G$) and $\lambda$,  $\mu$ are computed from the Young's modulus $E$ and the Poisson  ratio $\nu$ by the following formulas:
\begin{align}\label{3.1}
	\lambda=\frac{E\nu}{(1+\nu)(1-2\nu)},\qquad\mu=G=\frac{E}{2(1+\nu)}.
\end{align}

As for the problem (\ref{1.1})-(\ref{1.2}), there are many research results of PDE analysis and numerical analysis. For example, Phillips and Wheeler propose and analyze linear poroelasticity model for the continuous-in-time and discrete-in-time cases in \cite{B14,B15,B41}, later, they point out that there exist locking phenomenon in computation if $T$ is enough small in \cite{Phillips2009}. Yi develops a nonconforming mixed finite element method to solve Biot consolidation model in \cite{B42}. 
Feng, Ge and Li propose a multiphysics finite element method for a quasi-static poroelasticity model by a  multiphysics approach in \cite{B27}. 
Sun and Rui propose a coupling of the weak Galerkin method for solid displacement and the mixed finite element method of fluid phase in \cite{B43}.  
As for the brain edema,  which usually includes pathological elevated intracranial pressure, it can occur in specific parts or the whole brain, the high intracranial pressure will prevent blood from flowing to the brain, depriving it of the oxygen it needs to function properly, and the brain swelling also prevents other fluids from leaving the brain, making the swelling more worse, one can refer to \cite{B13,B32,B34}. In this paper, we uses the poroelasticity model to simulate brain edema  by the multiphysics finite element method(cf. \cite{B27}) and derive an optimal  $L^2$-norm error estimate of the  multiphysics finite element method for the poroelasticity model by introducing an auxiliary problem. The main contributions of the paper are listed  as follows:
\begin{enumerate}
	\renewcommand{\labelenumi}{(\theenumi)}
	\item 
	 It is the first time to derive an optimal  $L^2$-norm error estimate of the  multiphysics finite element method for the poroelasticity model by introducing an auxiliary problem.
	 
	\item  We apply the multiphysics finite element method to simulate the brain edema and carefully investigate the effects of the key physical parameters on brain edema. The numerical results show that the permeability $K$ has the biggest influence on intracranial pressure and tissue deformation, Young's modulus $E$ and Poisson ratio $\nu$ have little effect on the maximum value of intracranial pressure, but have great effect on the tissue deformation and the developing speed of brain edema.
\end{enumerate}

The remainder of this paper is organized as follows. In Section \ref{sec-3.2}, we reformulate the original model based on new variables and introduce the existence and uniqueness of weak solution, and derive the optimal convergence order of the fully discrete multiphysics finite element method  in $L^2$ and $H^1$ error estimate for $\mathbf{P}_2-P_1-P_1$ element pairs. In Section \ref{sec-3.4}, we show some numerical tests to verify the theoretical results and apply our method to simulate brain edema and investigate the effects of the key parameters  for brain edema. Finally, we draw conclusions to summarize the main results of the paper.

\section{ Optimal error estimate of multiphysics finite element method}\label{sec-3.2}
In this paper, we use the standard function space notation, which is defined in \cite{temam,B28}.
We denote $( \cdot , \cdot ) $ and $\langle \cdot , \cdot \rangle $  by the inner products of $L^2( \Omega) $ and $L^2( \partial\Omega)$ respectively.

Introduce new variables
\begin{align}\label{1113-1}
	q:= \operatorname{div}\mathbf{u},\quad \eta:=c_{0}p+\alpha q,\quad \xi:=\alpha p -\lambda q.
\end{align}
It is easy to check that
\begin{equation}\label{3.4}
	p=\kappa_{1} \xi + \kappa_{2} \eta, \qquad q=\kappa_{1} \eta-\kappa_{3} \xi,
\end{equation}
where
\begin{equation}\label{3.5}
	\kappa_{1}:= \frac{\alpha}{\alpha^{2}+\lambda c_{0}},
	\quad \kappa_{2}:=\frac{\lambda}{\alpha^{2}+\lambda c_{0}}, \quad 
	\kappa_{3}:=\frac{c_{0}}{\alpha^{2}+\lambda c_{0}}.
\end{equation}
Using (\ref{1113-1}), the problem (\ref{1.1})-(\ref{1.2}) can be rewritten as:
\begin{align}
	-\mu\operatorname{div}\varepsilon(\mathbf{u})+\nabla\xi=\mathbf{f} &\quad \mathrm{in~} \Omega_T, \label{3.6}\\
	\kappa_3\xi+\operatorname{div}\mathbf{u}=\kappa_1\eta &\quad \mathrm{in~}\Omega_T, \label{3.7}\\
	\eta_t- \frac{1}{\mu_f}\operatorname{div}[K(\nabla(\kappa_1\xi+\kappa_2\eta)-\rho_f \mathbf{g})]=\phi &\quad \mathrm{in~} \Omega_T. \label{3.8}
\end{align}
The boundary and initial conditions (\ref{1.4})-(\ref{1.6}) can be rewritten as:
\begin{align}
	\sigma(\mathbf{u})\mathbf{n}-\alpha (\kappa_1\xi+\kappa_2\eta)\mathbf{n}=\mathbf{f}_1& \quad \mathrm{on~} \partial\Omega_T,\label{3.9}\\
	-\frac{K}{\mu_f}(\nabla(\kappa_1\xi+\kappa_2\eta)-\rho_f \mathbf{g})\mathbf{n}=\phi_1& \quad \mathrm{on~} \partial\Omega_T,\label{3.10}\\
	\mathbf{u}=\mathbf{u_0},~~p=p_0& \quad \mathrm{in~}  \Omega\times\{t=0\}.\label{3.11}
\end{align}

In order to ensure the existence and uniqueness of solution to the problem (\ref{3.6})-(\ref{3.8}) with pure Neumann boundary conditions, we introduce the  infinitesimal rigid motions space 
\begin{align}
	\mathbf{RM}:=\{\mathbf{r}\mid\mathbf{r}=\mathbf{a}+\mathbf{b}\times \mathbf{x},~\mathbf{a},\mathbf{b},\mathbf{x}\in\mathbb{R}^d\}.
\end{align}
From \cite{temam,B22,B23} we known that $\mathbf{RM}$ is the kernel of the strain operator $\varepsilon$, that is, $\mathbf{r}\in \mathbf{RM}$ if and only if $\varepsilon(\mathbf{r})=0$. Hence, we have
\begin{align}
	\varepsilon(\mathbf{r})=0,\quad\div\mathbf{r}=0\qquad\forall\mathbf{r}\in\mathbf{RM}.
\end{align}
Moreover, we define $\mathbf{H}_\bot^1( \Omega)$, which is the subspace of  $\mathbf{H}^1(\Omega)$ and  is orthogonal to $\mathbf{RM}$, that is,
\begin{align}
	\mathbf{H}_\perp^1(\Omega)&:=\{\mathbf{v}\in\mathbf{H}^1(\Omega);\:(\mathbf{v},\mathbf{r})=0\:~~\forall\mathbf{r}\in\mathbf{RM}\}.
\end{align}
Next, we introduce an important result from Ref. \cite{B24} as follows: 
\begin{lemma}
	Let $U\in \mathbb{R}^{n}$ be a convex domain with diameter $\gamma$. Then
	\begin{align}\label{3-25-1}
		\|\phi\|_{L^2(U)}\leq \frac{\gamma}{\pi}\|
		\nabla\phi\|_{L^2(U)}
	\end{align}
	for all $\phi(x)\in H^1(U)$ satisfying $\int_{U}\phi(x)\mathrm{~dx}=0$.
\end{lemma}

Throughout the paper, we assume that $C$ may represent different positive constants in different places.
\begin{definition}
	Let $\mathbf{u}_0\in \mathbf{H}^1$, $\mathbf{f} \in \mathbf{L}^2(\Omega)$, $\mathbf{f}_1 \in \mathbf{L}^2(\partial\Omega)$, $p_0\in L^2(\Omega)$, $\phi\in L^2(\Omega)$, and  $\phi_1\in L^2(\partial\Omega)$. Assume that $(\mathbf{f},\mathbf{v})+\langle\mathbf{f}_1,\mathbf{v}\rangle=0\ \forall\mathbf{v}\in\mathbf{RM}$. Given $T>0$, a 3-tuple $(\mathbf{u}, \xi, \eta)$ with
	\begin{align*}
		\mathbf{u}\in L^{\infty}\left(0,T;\mathbf{H}_{\perp}^{1}(\Omega)\right), \ \xi\in L^{2}\left(0,T;H^{1}(\Omega)\right),\   
		\eta\in L^{\infty}\left(0,T;H^{1}(\Omega)\right)\cap H^{1}\left(0,T;H^{-1}(\Omega)\right) 
	\end{align*}
	is called a weak solution to \eqref{3.6}-\eqref{3.8} if there hold for almost every $t\in (0,T)$
	\begin{align}
		\mu(\varepsilon(\mathbf{u}),\varepsilon(\mathbf{v}))-(\xi,\mathop{\operatorname{div}}\mathbf{v}) =(\mathbf{f},\mathbf{v})+\langle\mathbf{f}_{1},\mathbf{v}\rangle \quad& \forall\mathbf{v}\in\mathbf{H}^{1}(\Omega), \label{2.15}\\
		\kappa_{3}(\xi,\varphi)+(\operatorname{div}\mathbf{u},\varphi) =\kappa_1(\eta,\varphi)\quad&\forall\varphi\in L^2(\Omega), \label{2.16}\\
		(\eta_{t},\psi)+\frac{1}{\mu_{f}}\big(K(\nabla(\kappa_{1}\xi+\kappa_{2}\eta) -\rho_f\mathbf{g}),\nabla\psi) 
		=(\phi,\psi)+\langle\phi_1,\psi\rangle\quad&\forall\psi\in H^1(\Omega), \label{2.17}\\
		\mathbf{u}(0)=\mathbf{u}_{0},\quad p(0) =p_{0},\quad
		q(0)=q_0:=\operatorname{div}\mathbf{u}_0,&\nonumber\\
		\eta(0)=\eta_0: =c_0p_0+\alpha q_0,\quad\xi(0)=\xi_0:=\alpha p_0-\lambda q_0. &\nonumber
	\end{align}
\end{definition}
\begin{theorem}
	Let $\mathbf{u}_0 \in\mathbf{H}^1(\Omega)$,
	$\mathbf{f}\in\mathbf{L}^2(\Omega)$, $\mathbf{f}_1\in\mathbf{L}^2(\partial\Omega)$, $p_0\in L^2(\Omega)$, $\phi\in L^2(\Omega)$, and $\phi_1\in L^2(\partial\Omega)$. Suppose $(\mathbf{f},\mathbf{v})+\langle\mathbf{f}_1,\mathbf{v}\rangle=0$ ~$\forall~\mathbf{v}\in\mathbf{RM}$. Then there exists a unique solution to the problem  \eqref{3.6}-\eqref{3.11}.
\end{theorem}
\begin{proof}
	For the detailed proof of the existence and uniqueness of the problem \eqref{3.6}-\eqref{3.11}, one can refer to \cite{B27}.
\end{proof}

Let $\mathcal{T}_h$ be a quasi-uniform triangulation or rectangular partition of $\Omega$ with mesh size $h$, and $\bar{\Omega}=\bigcup_{K\in\mathcal{T}_h}\bar{K}$.
The time interval $(0, T)$ is divided into $N$ equal intervals, denoted by $[t_{n-1},t_n], n=1,2,\cdots, N$ and $\Delta t=\frac{T}{N}$, $\mathrm{then~}t_n=n\Delta t$. 

Next, we define the following spaces of piecewise polynomials
\begin{align*}
\mathbf{R}_h &=\bigl\{\bv_h\in \bC^0(\overline{\Ome});\,
\bv_h|_K\in \bP_2(K)~~\forall K\in \cT_h \bigr\}, \\
R_h &=\bigl\{\varphi_h\in C^0(\overline{\Ome});\, \varphi_h|_K\in P_1(K)
~~\forall K\in \cT_h \bigr\},
\end{align*}
where $\mathbf{P}_2(K)$ and $P_1(K)$ are quadratic and linear polynomial spaces on $K$, respectively. Also, we introduce the function space
\begin{align*}
	L_0^2(\Omega):=\{q\in L^2(\Omega); (q,1)=0\},
\end{align*}
and define the mixed finite element pair $(\mathbf{X}_h,M_h,W_h)$ for variable $(\mathbf{u},\xi,\eta)$ as follows
\begin{align}
	\mathbf{X}_h=\mathbf{R}_h\cap \mathbf{H}^1(\Omega), \quad 
	M_h = R_h\cap L^2(\Omega).
\end{align}

The finite element approximation space $W_h$ for $\eta$ can be chosen independently, any piecewise polynomial space is acceptable provided that
$M_h \subseteq W_h\subseteq L^2(\Omega)$. 
The most convenient choice is $W_h =M_h$, which
will be adopted in the remainder of this paper. 

Moreover, we define
\begin{align}
	\mathbf{V}_h=\{\mathbf{v}_h\in\mathbf{X}_h;\mathrm{~}(\mathbf{v}_h,\mathbf{r})=0~~\forall\mathbf{r}\in\mathbf{R}\mathbf{M}\},
\end{align}
it is easy to check that $\mathbf{X}_h=\mathbf{V}_h\bigoplus\mathbf{RM}$. 
Note that the spaces $\mathbf{V}_h \times M_h$ are usually referred to as the generalized Taylor-Hood elements and satisfy the discrete inf-sup condition(cf. \cite{B35,B26})
\begin{align}
	\|q_h\|_{L^2(\Omega)}\leq C\sup_{\mathbf{0}\neq\mathbf{v}_h\in \mathbf{V}_h}\frac{(q_h, \div \mathbf{v}_h)}{\|\mathbf{v}_h\|_{H^1(\Omega)}}\qquad\forall q_h\in M_{0h}:=M_h\cap L_0^2(\Omega),
\end{align}
for some constant $C>0$.

\begin{algorithm}{Multiphysics finite element method (MFEM)} \label{alg2.1}
\begin{itemize}
\item[(i)]
Compute $\bu^0_h\in \bV_h$, $\xi_h^0\in M_h$ and $\eta^0_h\in W_h$ by
\begin{alignat*}{3}
\bu^0_h &=\mathcal{R}_h\bu_0, \quad p^0_h =\mathcal{Q}_hp_0, \quad
 q^0_h =\mathcal{Q}_hq_0 \ (q_0 =\operatorname{div} \bu_0), \\
\eta^0_h &=c_0p^0_h+\alpha q^0_h, \quad
\xi_h^0 =\alpha p_h^0 -\lambda q_h^0,
\end{alignat*}
where $\mathcal{R}_h$ and
$\mathcal{Q}_h$ are defined by (\ref{3.19}) and (\ref{3.24}).

\item[(ii)] For $n=0,1,2, \cdots$,  do the following two steps.

{\em Step 1:} Solve for $(\bu^{n+1}_h,\xi^{n+1}_h,\eta^{n+1}_h)\in \bV_h\times M_h\times W_h$ such that
\begin{align}
	\mu(\varepsilon(\mathbf{u}_{h}^{n+1}), \varepsilon(\mathbf{v}_h) )
	-( \xi^{n+1}_h, \operatorname{div} \mathbf{v}_h )
	=({\mathbf{f}},{\mathbf{v}_{h}})
	+\langle{\mathbf{f}_{1}},{\mathbf{v}_h}\rangle,&\quad\forall{\mathbf{v}_{h}}\in{\mathbf{V}_{h}},\label{3.14}\\
	\kappa_{3}(\xi_{h}^{n+1},{\varphi_{h}})
	+(\operatorname{div} \mathbf{u}_{h}^{n+1},{\varphi_{h}})=
	k_{1}(\eta_{h}^{n+\theta},{\varphi_{h}}),&\quad\forall \varphi_{h}\in M_{h},\label{3.15}\\
	({d_{t}\eta_{h}^{n+1}},\psi_{h})+\frac{1}{\mu_{f}}({K(\nabla(\kappa_{1}\xi_{h}^{n+1} +\kappa_{2}\eta_{h}^{n+1})-{\rho_{f}}\mathbf{g})},{\nabla\psi_{h}})&\nonumber\\
	=(\phi,\psi_{h})+\langle{\phi_{1}},\psi_{h}\rangle,&\quad\forall \psi_{h}\in M_{h}.\label{3.16}
\end{align}

{\em Step 2:} Update $p^{n+1}_h$ and $q^{n+1}_h$ by
\begin{eqnarray}
p^{n+1}_h=\kappa_1\xi^{n+1}_h +\kappa_2\eta^{n+\theta}_h, \quad
q^{n+1}_h=\kappa_1\eta^{n+1}_h-\kappa_3\xi^{n+1}_h,\label{3.17}
\end{eqnarray}
where $\theta=0$ or $1$.
\end{itemize}
\end{algorithm}
\begin{remark}
	When $\theta =1$, Algorithm \ref{alg2.1} is coupled; when $\theta =0$, Algorithm \ref{alg2.1} is decoupled. 
	For the stability analysis of Algorithm \ref{alg2.1}, one can refer to \cite{B27}.
\end{remark}

Next, we analyze the optimal error estimates of Algorithm \ref{alg2.1}. To do that, 
for any $\varphi\in L^2(\Omega)$ we define its $L^2$-projection $\mathcal{Q}_h: L^2\rightarrow W_h$ as
\begin{align}\label{3.19}
	\bigl( \mathcal{Q}_h\varphi, \psi_h  \bigr)=\bigl( \varphi, \psi_h  \bigr) \qquad \forall\psi_h\in W_h.
\end{align}

For any $\varphi\in H^1(\Omega)$, we define the elliptic projection $\mathcal{S}_h\varphi$ by
\begin{alignat}{2}\label{3.21}
	\bigl(K\nabla\mathcal{S}_h\varphi, \nabla\varphi_h\bigr) &=\bigl(K\nabla\varphi, \nabla\varphi_h\bigr)
	&& \qquad \forall \varphi_h\in W_h,\\
	\bigl(\mathcal{S}_h\varphi, 1\bigr) &=\bigl(\varphi, 1\bigr).&&\label{3.22}
\end{alignat}

For any $\bv\in \bH^1_\perp(\Omega)$, we define the elliptic projection $\mathcal{R}_h\bv$ by
\begin{alignat}{2}\label{3.24}
	\bigl(\varepsilon(\mathcal{R}_h\bv), \varepsilon(\bw_h)\bigr)
	=\bigl(\varepsilon(\bv), \varepsilon(\bw_h)\bigr) \quad \forall\bw_h\in \bV_h.
\end{alignat}

From \cite{B28}, we know that for any $\varphi\in H^s(\Omega) (s\geq1)$ the projection operator $\mathcal{Q}_h$ satisfies
\begin{align}\label{3.20}
\|\mathcal{Q}_h\varphi-\varphi\|_{L^2(\Ome)}+h\| \nabla(\mathcal{Q}_h\varphi
-\varphi) \|_{L^2(\Ome)}\leq Ch^\ell\|\varphi\|_{H^\ell(\Ome)}, \quad \ell=\min\{2, s\};
\end{align}
for any $\varphi\in H^s(\Omega) (s>1)$ the projection operator $\mathcal{S}_h$ satisfies 
\begin{align}\label{3.23}
\|\mathcal{S}_h\varphi-\varphi\|_{L^2(\Ome)}+h\| \nabla(\mathcal{S}_h\varphi-\varphi) \|_{L^2(\Ome)}
\leq Ch^\ell\|\varphi\|_{H^\ell(\Ome)}, \quad \ell=\min\{2, s\};
\end{align}
for any $\bv\in \bH^1_\perp(\Omega)\cap \bH^s(\Omega) (s>1)$ the projection $\mathcal{R}_h\bv$ satisfies
\begin{align}\label{3.25}
	\|\mathcal{R}_h\bv-\bv\|_{L^2(\Ome)}+h\| \nabla(\mathcal{R}_h\bv-\bv) \|_{L^2(\Ome)}
	\leq Ch^m\|\bv\|_{H^m(\Ome)}, m=\min\{3, s\}.
\end{align} 

\begin{theorem}\label{thm3.5}
	The solution of Algorithm \ref{alg2.1} satisfies the following error estimates:
	\begin{align}
		&\max_{0\leq n\leq N} \Bigl[ \sqrt{\mu}\|\nabla(\mathbf{u}(t_{n})-\mathbf{u}_{h}^{n})\|_{L^2(\Omega)}
		+\sqrt{\kappa_2} \|\eta(t_n)-\eta_h^{n-1+\theta}\|_{L^2(\Ome)}\nonumber\\
		&\hskip .5in
		+\sqrt{\kappa_3} \|\xi(t_n)-\xi_h^n\|_{L^2(\Ome)} \Bigr]
		\leq C_1(T) \Delta t +C_2(T)h^2,\label{2.32} \\
		&\bigg[ \Delta t \sum_{n=0}^N \frac{K}{\mu_f} \|\nabla(p(t_n)-p_h^n) \|_{L^2(\Ome)}^2 \bigg]^{\frac12}
		\leq C_1(T) \Delta t +C_2(T)h, \label{2.33}
	\end{align}
	provided that $\Delta t=O(h^2)$ if $\theta=0$ and $\Delta t>0$ if $\theta=1$. Here
	\begin{align*}
		C_1(T):=&C(\|q_t\|_{L^2(0,T;L^2(\Omega))}+\|q_{tt}\|_{L^2(0,T;H^{-1}(\Omega))}),\\
		C_2(T):=&C(\|\xi\|_{L^{\infty}(0,T;H^2(\Ome))}+\|\xi_t\|_{L^{\infty}(0,T;H^2(\Ome))}+\|\operatorname{div}\mathbf{u}_t\|_{L^{\infty}(0,T;H^2(\Ome))} \\
		&
		+\|\eta\|_{L^{\infty}(0,T;H^2(\Ome))}
		+\|\nabla\bu\|_{L^{\infty}(0,T;H^2(\Ome))}).
	\end{align*}
\end{theorem}
As for the detailed proof of Theorem \ref{thm3.5}, one can refer to \cite{B27}.

Next, we give the optimal convergence order in $L^2$-norm error estimate for the displacement  $\mathbf{u}$. 

\begin{theorem}\label{thm2.5}
The numerical solution of Algorithm \ref{alg2.1} satisfies the following error estimates
	\begin{eqnarray}\label{3.30}
		\|\mathbf{u}(t_{n})-\mathbf{u}_h^{n}\|_{L^{2}(\Omega)} \leq Ch^{3}.
	\end{eqnarray}
\end{theorem}

\begin{proof}
Define a bilinear form  
\begin{align}
	B((\mathbf{u}, \xi) ;(\mathbf{v},q))
	=\mu(\varepsilon(\mathbf{u}),\varepsilon(\mathbf{v}))
	-(\xi, \nabla \cdot \mathbf{v})
	+(q, \nabla \cdot \mathbf{u})
	+k_{3}(\xi, q). 
\end{align}
Using (\ref{2.15}) and (\ref{2.16}), we have
	\begin{align}
		&B((\mathbf{u}, \xi) ;(\mathbf{u}, \xi))=\mu|\mathbf{u}|_{H^{1}( \Omega)}+k_{3}\|\xi\|_{L^2(\Omega)}, \\
		&B((\mathbf{u}, \xi) ;(\mathbf{v}, q))
		\leqslant C\left(|\mathbf{u}|_{H^1(\Omega)}+\|\xi\|_{L^2(\Omega)}\right)\left(|\mathbf{v}|_{H^1(\Omega)}+\|q\|_{L^2(\Omega)}\right).\label{3.18}
	\end{align}
	
Next, we introduce the following auxiliary problem: find $(\mathbf{w},z)\in H^1(\Omega)\cap L^2_0(\Omega)$ such that 
\begin{align}
	-\mu\operatorname{div}\varepsilon(\mathbf{w})-\nabla z&=\mathbf{u}-\mathbf{u}_{h}^{n} \quad \mathrm{in~} \Omega,\label{3.31} \\
	\nabla\cdot\mathbf{w}-\kappa_3 z&=0  \quad \mathrm{in~} \Omega, \label{3.32}\\
	\mu\varepsilon(\mathbf{w})\mathbf{n}+z\mathbf{n}&=\mathbf{0}\quad \mathrm{on~} \partial\Omega.
\end{align}

From \cite{temam}, we know that
\begin{align}\label{2025-4-23}
	\|\mathbf{w}\|_{H^2(\Omega)}+\|z\|_{H^1(\Omega)}\leq C\|\mathbf{u}-\mathbf{u}_{h}^{n}\|_{L^2(\Omega)}.
\end{align}	

Multiplying \eqref{3.31} by $\mathbf{u}-\mathbf{u}_{h}^{n}$ and \eqref{3.32} by $\xi-\xi_{h}^{n}$, we get
\begin{align*}
	\|\mathbf{u}-\mathbf{u}_{h}^{n}\|_{L^{2}(\Omega)}^{2}
	&=(\mathbf{u}-\mathbf{u}_{h}^{n},\mathbf{u}-\mathbf{u}_{h}^{n})\\
	&=(-\mu\operatorname{div}\varepsilon(\mathbf{w})-\nabla z,\mathbf{u}-\mathbf{u}_{h}^{n})\\
	&=(-\mu\operatorname{div}\varepsilon(\mathbf{w}),\mathbf{u}-\mathbf{u}_{h}^{n})
	-(\nabla z,\mathbf{u}-\mathbf{u}_{h}^{n})\\
	&=\mu(\varepsilon(\mathbf{w}), \varepsilon(\mathbf{u}-\mathbf{u}_{h}^{n}))
    +(z, \nabla\cdot(\mathbf{u}-\mathbf{u}_{h}^{n}))\\
    &=\mu(\varepsilon(\mathbf{w}), \varepsilon(\mathbf{u}-\mathbf{u}_{h}^{n}))
    +(z, \nabla\cdot(\mathbf{u}-\mathbf{u}_{h}^{n}))\\
	&\quad-(\nabla\cdot \mathbf{w},(\xi-\xi_{h}^{n}))+k_{3}(z,(\xi-\xi_{h}^{n}))\\
	&=B((\mathbf{u}-\mathbf{u}_{h}^{n}, \xi-\xi_{h}^{n}) ;(\mathbf{w}, z))\\
	&=B((\mathbf{u}-\mathbf{u}_{h}^{n}, \xi-\xi_{h}^{n}) ;(\mathbf{w}-\mathcal{R}_{h}\mathbf{w}, z-\mathcal{Q}_hz))\\
	&\quad+B((\mathbf{u}-\mathbf{u}_{h}^{n}, \xi-\xi_{h}^{n}) ;(\mathcal{R}_{h}\mathbf{w}, \mathcal{Q}_hz)).
\end{align*}

Using \eqref{3.18}, \eqref{2.32}, \eqref{3-25-1} and \eqref{2025-4-23}, we get
\begin{align*}
	\|\mathbf{u}-\mathbf{u}_{h}^{n}\|_{L^{2}(\Omega)}^{2}
	&\leq C\left(|\mathbf{u}-\mathbf{u}_{h}^{n}|_{H^1(\Omega)}+\|\xi-\xi_{h}^{n}\|_{L^{2}(\Omega)}\right)\left(|\mathbf{w}-\mathcal{R}_{h}\mathbf{w}|_{H^1(\Omega)}+\|z-\mathcal{Q}_hz\|_{L^{2}(\Omega)}\right)\\
	&\quad+\kappa_1\left(\eta-\eta_{h}^{n-1+\theta},\mathcal{Q}_hz-z\right)
	+\kappa_1\left(\eta-\eta_{h}^{n-1+\theta},z\right)\\
	&\leq C_2(T)h^2\left(Ch\|\mathbf{w}\|_{H^2(\Omega) }+Ch\|z\|_{H^1(\Omega)}\right)
	+C_2(T)h^2(Ch\|z\|_{H^{1}(\Omega)}+\|z\|_{L^2(\Omega)})\\
	&=C_2(T)h^2\left(Ch\|\mathbf{w}\|_{H^2(\Omega) }+Ch\|z\|_{H^1(\Omega)}\right)
	+C_2(T)h^2(Ch\|z\|_{H^{1}(\Omega)}\\
	&\quad+(\sum_{K\in\mathcal{T}_h}\|z\|^2_{L^2(K)})^{\frac{1}{2}})\\
	&\leq C_2(T)h^2\left(Ch\|\mathbf{w}\|_{H^2(\Omega) }+Ch\|z\|_{H^1(\Omega)}\right)
	+C_2(T)h^2(Ch\|z\|_{H^{1}(\Omega)}\\
	&\quad+(\sum_{K\in\mathcal{T}_h}\frac{h^2}{\pi^2}\|\nabla z\|^2_{L^2(K)})^{\frac{1}{2}})\\
	&= C_2(T)h^2\left(Ch\|\mathbf{w}\|_{H^2(\Omega) }+Ch\|z\|_{H^1(\Omega)}\right)
	+C_2(T)h^2(Ch\|z\|_{H^{1}(\Omega)}\\
	&\quad+(\frac{h^2}{\pi^2}\|\nabla z\|^2_{L^2(\Omega)})^{\frac{1}{2}})\\
	&\leq Ch^3\left(\|\mathbf{w}\|_{H^2(\Omega) }+\|z\|_{H^1(\Omega)}\right)
	+Ch^2(h\|z\|_{H^{1}(\Omega)}+h\|\nabla z\|_{L^2(\Omega)})\\
	&\leq Ch^3\left(\|\mathbf{w}\|_{H^2(\Omega) }+\|z\|_{H^1(\Omega)}\right)
	+Ch^3\|z\|_{H^{1}(\Omega)}\\
	&\leq Ch^{3}\|\mathbf{u}-\mathbf{u}_{h}^{n}\|_{L^{2}(\Omega)},
\end{align*}
which implies that (\ref{3.30}) holds. The proof is complete.
\end{proof}

\section{Numerical tests}\label{sec-3.4}
In this section, we show some numerical tests to verify the theoretical results and and simulate brain edema by Algorithm \ref{alg2.1}. 
\begin{example}
	Let $\Omega= (0,1)\times(0,1)$, $T=1$, $\Gamma_{1}= \{ (0,y);~0\leq y\leq1 \}$, $\Gamma_{2}= \{(1,y);~0\leq y\leq1 \}$, $\Gamma_{3}= \{(x,1);~0\leq x\leq1 \}$, $\Gamma_{4} = \{(x,0);~0\leq x\leq1\}$.
\end{example}

The body force $\mathbf{f}$ and source term $\phi$ are given by
\begin{align*}
	\mathbf{f}=&
	\begin{pmatrix}
		-t(\lambda+\mu)+\alpha\cos(x+y)\mathrm{e}^t\\
		-t(\lambda+\mu)+\alpha\cos(x+y)\mathrm{e}^t
	\end{pmatrix},\\
	\phi = & (c_0+2\frac{K}{\mu_{f}})\sin(x+y)\mathrm{e}^t+\alpha(x+y).
\end{align*}

The initial and boundary conditions are
\begin{align*}
	p(x,y,t) = \sin(x+y)\mathrm{e}^t  &\qquad\mbox{on }\partial\Omega_T,\\
	u_1(x,y,t)= \frac{t}{2}x^{2} &\qquad\mbox{on }\Gamma_j\times [0,T],\, j=1,3,\\
	u_2(x,y,t) = \frac{t}{2}y^{2} &\qquad\mbox{on }\Gamma_j\times [0,T],\, j=,2,4,\\
	\sigma(\mathbf{u})\bf{n}-\alpha p\bf{n} = \mathbf{f}_1 &\qquad \mbox{on } \p\Ome_T\backslash\Gamma_j,\\
	\mathbf{u}(x,y,0) = \mathbf{0},  \quad p(x,y,0) =\sin(x+y) &\qquad\mbox{in } \Omega,
\end{align*}
where the function $\mathbf{f}_1$ can be obtained by \eqref{1.4}. It is easy to check that the exact solutions are
\begin{align*}
	\mathbf{u}(x,y,t)=
	\begin{pmatrix}
		u_1(x,y,t)\\
		u_2(x,y,t)
	\end{pmatrix}
	=\frac{t}{2}
	\begin{pmatrix}
		 x^2\\
		 y^2
	\end{pmatrix},\qquad
	p(x,y,t) =\sin(x+y)\mathrm{e}^{t}.
\end{align*}

The values of physical parameters are given in Table \ref{table_1} and the other variables $\kappa_1$ as follows:
\begin{table}[H]
	\begin{center}
		\caption{Values of physical parameters}\label{table_1}
		\begin{tabular}{ccc}
			\hline
			Parameters 	& Description		& Values \\
			\hline
			$E$ 			& Young's modulus 							& 599.6\\
			$\nu$ 		& Poisson ratio 							& 0.499\\
			$\lambda$ 		& Lam$\acute{e}$ constant  							& 1e6\\
			$\mu$ 		& Lam$\acute{e}$ constant  							& 2e2\\
			$\alpha$ 	& Biot-Willis constant 						& 0.6\\
			$c_0$  		& Constrained specific storage coefficient 	& 0.8\\
			$K$ 			& Permeability tensor 		& $\mathbf{I}$\\
			$\mu_f$ 	& Viscosity coefficient 						& 1\\
			\hline
		\end{tabular}
	\end{center}
\end{table}

\begin{table}[H]
	\begin{center}
		\caption{Spatial errors and convergence orders of $\mathbf{u}$ and $p$ with $\Delta t=h^2$ }\label{tab11}
		\begin{tabular}{l c c c c c c c c}
			\hline
			$h$& $\|\mathbf{u}-\mathbf{u}_h\|_{L^2}$  &  order& $\|\mathbf{u}-\mathbf{u}_h\|_{H^1}$ & order& $\|p-p_h\|_{L^2}$ &order& $\|p-p_h\|_{H^1}$&order
			\\ \hline
		$1/2$ & 2.3917e-6& & 3.1786e-5 & & 0.0462& & 0.4446 & \\
		$1/4$ &1.8952e-7 &3.66 & 5.7760e-6&2.46 &0.0109 &2.08 & 0.2231&0.99 \\
		$1/8$ &1.6219e-8&3.55&  1.0243e-6&2.50 &0.0027&2.01&  0.1116&1.00 \\
		$1/16$ &1.5440e-9&3.39 &  1.8102e-7&2.50&6.643e-4&2.02 &  0.0558& 1.00
		\\ \hline
		\end{tabular}
	\end{center}
\end{table}

\begin{figure}[H]
	\centering
	\subfigure[]{
		\includegraphics[width=0.45\textwidth]{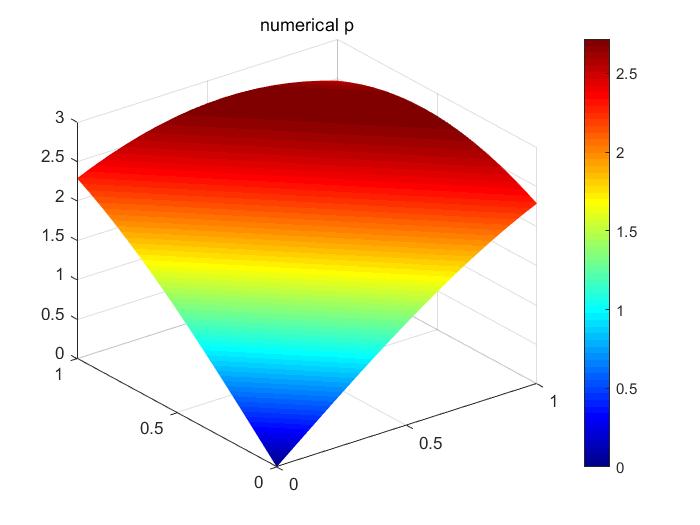}}
	\hfill
	\subfigure[]{
		\includegraphics[width=0.45\textwidth]{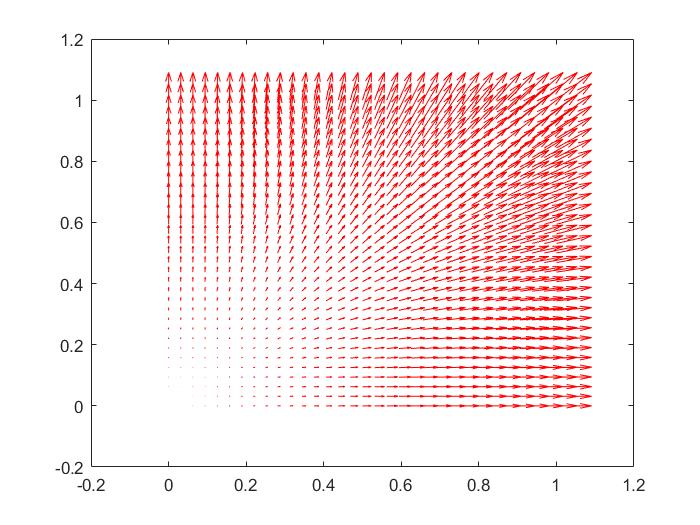}}
	\caption{(a) surface plot of the computed pressure $p$ and (b) arrow plot of the computed displacement $\mathbf{u}$ at $T$.}\label{figure_p11}
\end{figure}
\begin{figure}[H]
	\centering
	\subfigure[]{
		\includegraphics[width=0.45\textwidth]{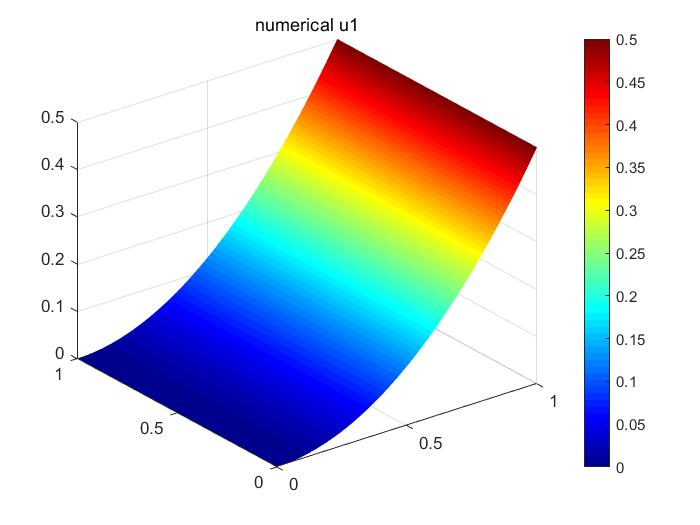}}
	\hfill
	\subfigure[]{
		\includegraphics[width=0.45\textwidth]{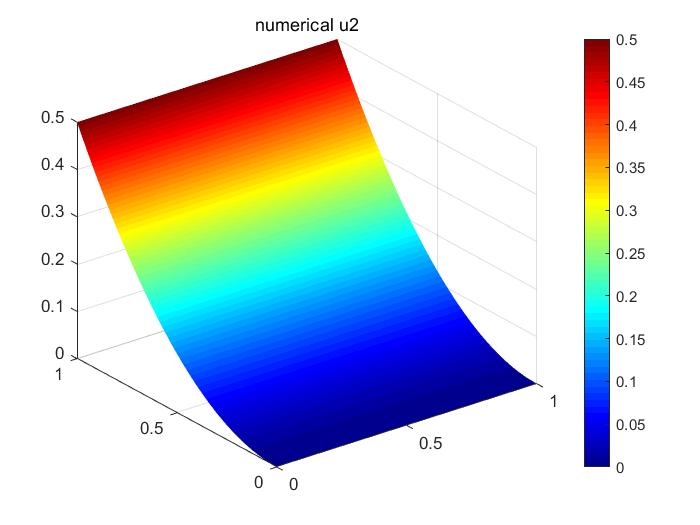}}
	\caption{Surface plot of the computed displacement (a) $u_1$ and (b) $u_2$ at the terminal time $T$, respectively.}\label{figure_p12}
\end{figure}

In this paper, the order of spatial error convergence is defined as 
\begin{align*}
	{\rm{order}}=\frac{\log(R(h)/R(\frac12h))}{\log2},
\end{align*}
where $R(h)$ is the spatial error with mesh $h$ .

Table \ref{tab11} displays the computed $L^2( \Omega)  $-norm and $H^1( \Omega) $-norm errors of $\mathbf{u}$, $p$ and the convergence orders with respect to $h$ at the terminal time $T$.
Evidently, the convergence orders are consistent with Theorem \ref{thm3.5} and Theorem \ref{thm2.5}. Figure \ref{figure_p11}-Figure \ref{figure_p12} show the surface plot of the computed  pressure $p$, displacement $u_1$, $u_2$ and the arrow plot of the computed displacement $u$ at the terminal time $T$  with mesh parameters $h= \frac{1}{16}$ and $\Delta t= h^2$, respectively, which coincide with the exact solution.

\begin{example}
Let $\Omega= (0,1)\times(0,1)$, $T=1$, $\Gamma_{1}= \{ (0,y);~0\leq y\leq1 \}$, $\Gamma_{2}= \{(1,y);~0\leq y\leq1 \}$, $\Gamma_{3}= \{(x,1);~0\leq x\leq1 \}$, $\Gamma_{4} = \{(x,0);~0\leq x\leq1\}$.
\end{example}

The body force $\mathbf{f}$ and source term $\phi$ are as follows
\begin{align*}
	\mathbf{f}=&
	\begin{pmatrix}
	(-2(\lambda+\mu)\pi^2t^2+\alpha\pi t)\cos(\pi x)\cos(\pi y)\\
	(2(\lambda+\mu)\pi^2t^2-\alpha\pi t)\sin(\pi x)\sin(\pi y)
	\end{pmatrix},\\
	\phi = & (c_0+4\alpha t+2\frac{K}{\mu_{f}}\pi^2 t)\sin(\pi x)\cos(\pi y).
\end{align*}

The initial and boundary conditions are
\begin{align*}
	p = t\sin(\pi x)\cos(\pi y)  &\qquad\mbox{on }\partial\Omega_T,\\
	u_1 = -t^2\cos(\pi x)\cos(\pi y) &\qquad\mbox{on }\Gamma_j\times [0,T],\, j=1,3,\\
	u_2 = t^2\sin(\pi x)\sin(\pi y) &\qquad\mbox{on }\Gamma_j\times [0,T],\, j=,2,4,\\
	\sigma\bf{n}-\alpha p\bf{n} = \mathbf{f}_1, &\qquad \mbox{on } \p\Ome_T\backslash\Gamma_j,\\
	\mathbf{u}(x,y,0) = \mathbf{0},  \quad p(x,y,0) =0 &\qquad\mbox{in } \Ome.
\end{align*}

It is easy to check that the exact solution are
\begin{align*}
	\mathbf{u}(x,y,t)=t^2
	\begin{pmatrix}
		-\cos(\pi x)\cos(\pi y)\\
		\sin(\pi x)\sin(\pi y)
	\end{pmatrix},\qquad
	p(x,y,t) = t\sin(\pi x)\cos(\pi y).
\end{align*}

The values of physical parameters are given in Table \ref{table_2}.
\begin{table}[H]
	\begin{center}
		\caption{Values of physical parameters}\label{table_2}
		\begin{tabular}{ccc}
			\hline
			Parameters 	& Description		& Values \\
			\hline
			$E$ 			& Young's modulus 							& 2.91e4\\
			$\nu$ 		& Poisson ratio 							& 0.454\\
			$\lambda$ 		& Lam$\acute{e}$ constant  							& 1e5\\
			$\mu$ 		& Lam$\acute{e}$ constant  							& 1e4\\
			$\alpha$ 	& Biot-Willis constant 						& 0.05\\
			$c_0$  		& Constrained specific storage coefficient 	& 0.88\\
			$K$ 			& Permeability tensor 		& $\mathbf{I}$\\
			$\mu_f$ 	& Viscosity coefficient						& 1\\
			\hline
		\end{tabular}
	\end{center}
\end{table}

\begin{table}[H]
	\begin{center}
		\caption{Spatial errors and convergence orders of $\mathbf{u} $ and $p$ with $\Delta t=h^2$ }\label{tab21}
		\begin{tabular}{l c c c c c c c c}
			\hline
			$h$& $\|\mathbf{u}-\mathbf{u}_h\|_{L^2}$  &  order& $\|\mathbf{u}-\mathbf{u}_h\|_{H^1}$ & order
			& $\|p-p_h\|_{L^2}$ &order& $\|p-p_h\|_{H^1}$&order 
			\\ \hline
			$1/4$ & 3.85e-2& & 0.9747 & & 0.0520& & 0.8534 & \\
			$1/8$ &2.40e-3 &4.00 & 0.1437&2.76 &0.0144 &1.85 & 0.4338&0.98 \\
			$1/16$ &1.74e-4&3.79&  0.0223&2.69 &3.70e-3&1.96 & 0.2178&0.99 \\
			$1/32$ &1.54e-5&3.50 &  0.0004&2.49&9.32e-4&1.99& 0.1090& 1.00\\ \hline
		\end{tabular}
	\end{center}
\end{table}

\begin{figure}[H]
	\centering
	\subfigure[]{
	\includegraphics[width=0.45\textwidth]{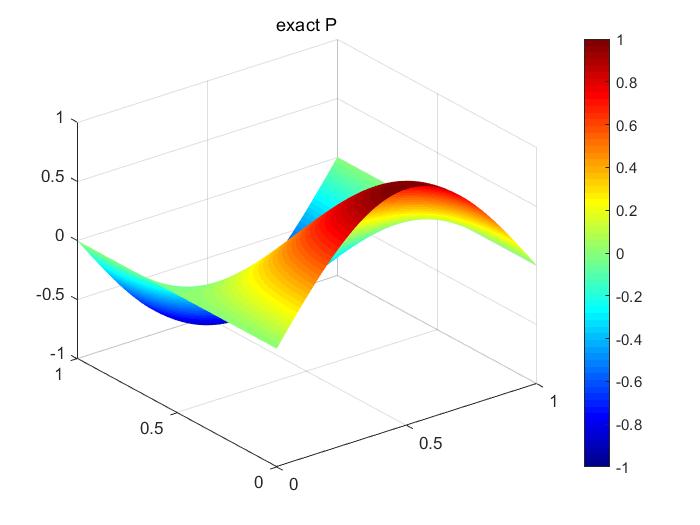}
}
	\hfill
	\subfigure[]{
	\includegraphics[width=0.45\textwidth]{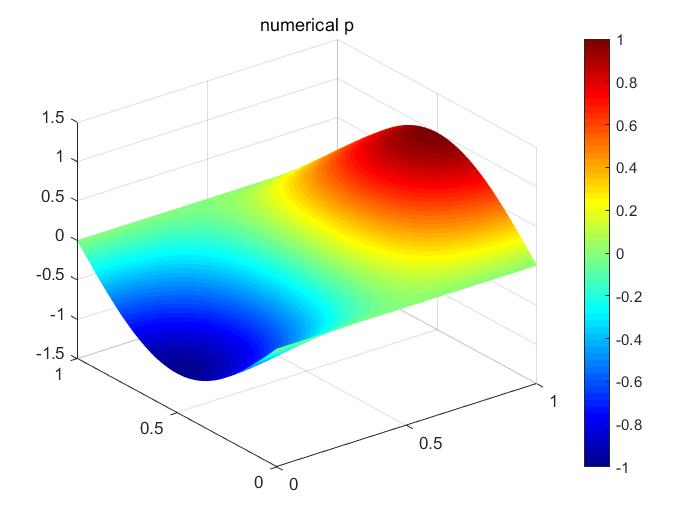}}
	\caption{The solution $p$ at the terminal time $T$: (a) the exact solution, (b) the numerical solution.}\label{f3.5}
\end{figure}
\begin{figure}[H]
	\centering
	\subfigure[]{
		\includegraphics[width=0.45\textwidth]{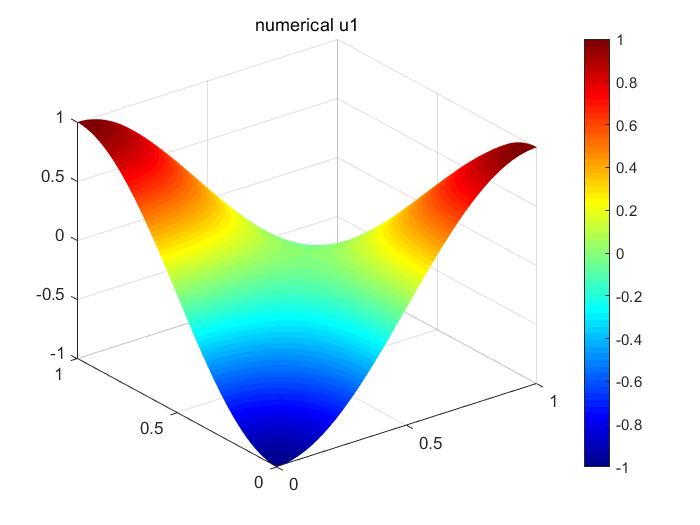}}
	\hfill
	\subfigure[]{
		\includegraphics[width=0.45\textwidth]{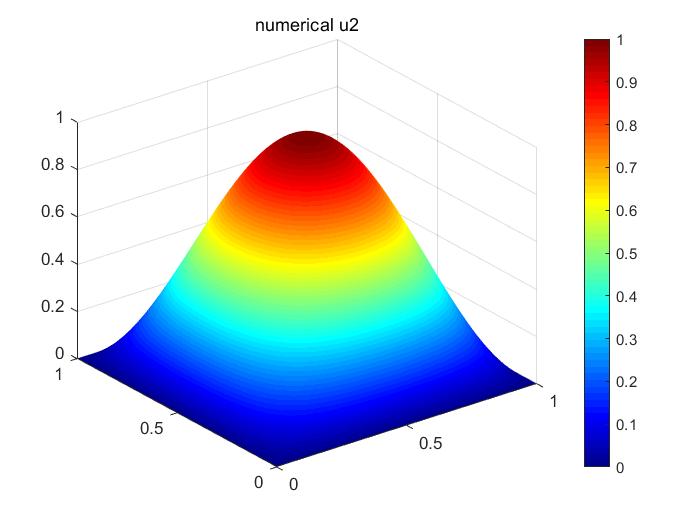}}
	\caption{Surface plots of the computed displacement (a) $u_1$   and (b) $u_2$  at the terminal time $T$, respectively.}\label{f3.7}
\end{figure}
\begin{figure}[H]
	\centering
	\includegraphics[width=8cm]{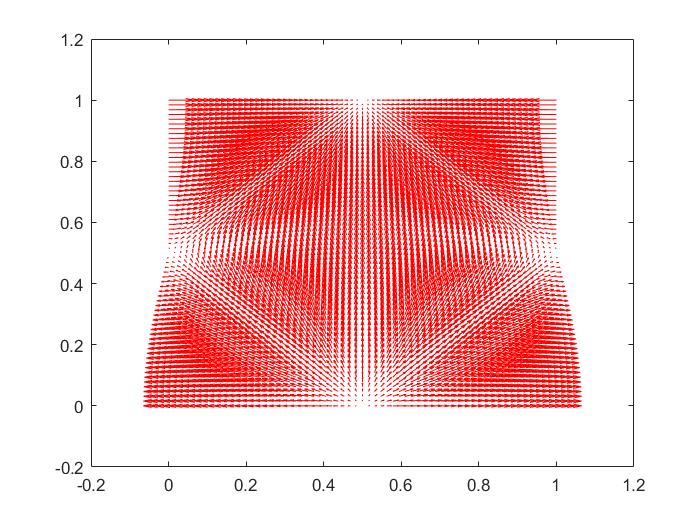}
	\caption{Arrow plot of the computed displacement $\mathbf{u}$ at $T$.}\label{f3.9}
\end{figure}

Table \ref{tab21} displays the computed $L^2( \Omega) $-norm and $H^1( \Omega)$-norm errors of $\mathbf{u}$, $p$ and the convergence orders with respect to $h$ at the terminal time $T$.
Evidently, the convergence orders are consistent with Theorem \ref{thm3.5} and Theorem \ref{thm2.5}. Figure \ref{f3.5}-Figure \ref{f3.9} show the surface plot of the computed pressure $p$, 
displacement $u_1$, $u_2$ and the arrow plot of the computed displacement $u$ at the terminal time $T$  with mesh size $h= \frac{1}{16}$ and $\Delta t= h^2$, respectively, which coincide with the exact solution.

\begin{example}
\textrm{Mesh~deformation.}
\end{example}
This test considers a scenario where mesh nodes are always attached to material points and deform with them. The initial mesh is a triangular discretization of a two-dimensional rectangular domain; once generated, this mesh is fixed as the sole initial configuration and is independent of time. In terms of implementation, the program adopts the Total Lagrangian formulation: it consistently uses the initial mesh configuration as the reference and determines the deformed state by computing the displacement field relative to the initial positions at each time instant. For a target time $t$, the program directly superimposes the computed displacement vector onto the initial nodal coordinates, starting from the initial mesh, to obtain the deformed mesh at that instant-without requiring stepwise updates from the previous configuration. This differs from the Updated Lagrangian formulation, where calculations are performed on the newly deformed mesh at each step. This strategy simplifies the computational procedure and is particularly suitable for poroelastic problems under the linear strain assumption, since the equilibrium equations can be established on the fixed initial configuration, eliminating the need to reassemble the stiffness matrix based on the updated mesh after each time step. Although the results for each target time are obtained directly, outputting the deformed meshes at multiple time points still fully captures the deformation history of the structure. 

Let $\Omega= (0,1.5)\times(0,1.5)$, $\Gamma_{1}= \{ (0,y);~0\leq y\leq1.5 \}$, $\Gamma_{2}= \{(1.5,y);~0\leq y\leq1.5 \}$, $\Gamma_{3}= \{(x,1.5);~0\leq x\leq1.5 \}$, $\Gamma_{4} = \{(x,0);~0\leq x\leq1.5\}$. $T=0.5$, $t\in [0,T].$
The body force $\mathbf{f}$ and source term $\phi$ are as follows
\begin{align*}
	\mathbf{f}=&
	\begin{pmatrix}
		(\frac{3}{2}\mu-\lambda)\pi^2 t\sin(\pi x)sin(\pi y)-2\alpha\pi t\sin(2\pi x)\cos(2\pi y)-\lambda t(2x-1)(2y-1)-\mu t y(y-1)
		\\
		-(\frac{\mu}{2}+\lambda)\pi^2t\cos(\pi x)\cos(\pi y)-2\alpha\pi t\cos(2\pi x)\sin(2\pi y)-2(\mu+\lambda)tx(x-1)-\mu ty(y-1)
	\end{pmatrix},\\
	\phi = & (c_0+\frac{8K}{\mu_{f}}\pi^2 t)\cos(2\pi x)\cos(2\pi y)+\alpha\pi\cos(\pi x)\cos(\pi y)+\alpha x(x-1)(2y-1).
\end{align*}

The initial and boundary conditions are
\begin{align*}
	p = t\cos(2\pi x)\cos(2\pi y)  &\qquad\mbox{on }\partial\Omega_T,\\
	u_1 = t\sin(\pi x)\sin(\pi y) &\qquad\mbox{on }\Gamma_j\times [0,T],\, j=1,2,\\
	u_2 = txy(x-1)(y-1) &\qquad\mbox{on }\Gamma_j\times [0,T],\, j=,1,2,\\
	\sigma\bf{n}-\alpha p\bf{n} = \mathbf{f}_1, &\qquad \mbox{on } \p\Ome_T\backslash\Gamma_j,\\
	\mathbf{u}(x,y,0) = \mathbf{0},  \quad 
	p(x,y,0) =0 &\qquad\mbox{in } \Ome.
\end{align*}

It is easy to check that the exact solution are
\begin{align*}
	\mathbf{u}(x,y,t)=t
	\begin{pmatrix}
		t\sin(\pi x)\sin(\pi y)\\
		txy(x-1)(y-1)
	\end{pmatrix},\qquad
	p(x,y,t) = t\cos(2\pi x)\cos(2\pi y).
\end{align*}

The term $\mathbf{f}_1$ can be obtained from the exact solution.
The values of physical parameters are given in Table \ref{table_test3_1}. 
The other physical variables $\lambda$, $\mu$, $\kappa_1$, $\kappa_2$, $\kappa_3$  can be calculated by \eqref{3.1} and \eqref{3.5}.
\begin{table}[H]
	\begin{center}
		\caption{Values of physical parameters}\label{table_test3_1}
		\begin{tabular}{ccc}
			\hline
			Parameters 	& Description		& Values \\
			\hline
			$E$ 			& Young's modulus 							& 1e7\\
			$\nu$ 		& Poisson ratio 							& 0.4\\
			$\alpha$ 	& Biot-Willis constant 						& 0.5\\
			$c_0$  		& Constrained specific storage coefficient 	& 0.5\\
			$K$ 			& Permeability tensor 		& $\mathbf{I}$e-9\\
			$\mu_f$ 	& Viscosity coefficient						& 1\\
			\hline
		\end{tabular}
	\end{center}
\end{table}

\begin{table}[H]
	\begin{center}
		\caption{Spatial errors and convergence orders of $\mathbf{u} $ and $p$ with $\Delta t=T/5$ }\label{table_test3_2}
		\begin{tabular}{l c c c c c c c c}
			\hline
			$h$& $\|\mathbf{u}-\mathbf{u}_h\|_{L^2}$  &  order& $\|\mathbf{u}-\mathbf{u}_h\|_{H^1}$ & order
			& $\|p-p_h\|_{L^2}$ &order& $\|p-p_h\|_{H^1}$&order 
			\\ \hline
			$1/16$ & 1.9496e-4& & 1.7407e-2 & & 6.2628e-3& &  1.0141 & \\
			$1/32$ &2.0572e-5 &  3.24  & 3.9035e-3 &  2.15  & 1.3067e-3 & 2.26 &  4.9501e-1 &  1.03 \\
			$1/64$ & 2.4098e-6 &  3.09  & 9.3729e-4 &  2.05 &  3.0912e-4 &  2.08  & 2.4589e-1 &  1.01 \\
			$1/128$ & 2.9415e-7 &  3.03 &  2.3086e-4  & 2.02  & 7.6149e-5 &  2.02 &  1.2274e-1  & 1.00\\ 
			\hline
		\end{tabular}
	\end{center}
\end{table}

\begin{figure}[H]
	\centering
	\subfigure[$t=0$]{
		\includegraphics[width=0.44\textwidth]{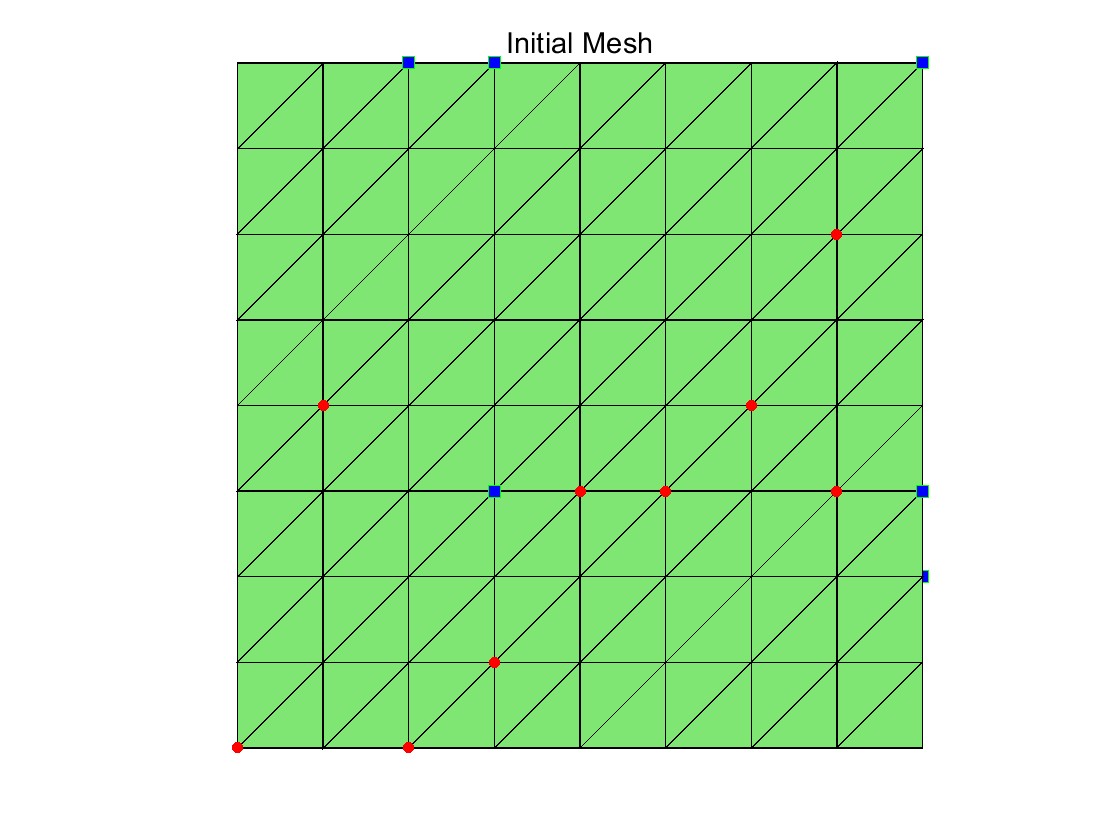}
	}
	\subfigure[$t=0.1$]{
		\includegraphics[width=0.44\textwidth]{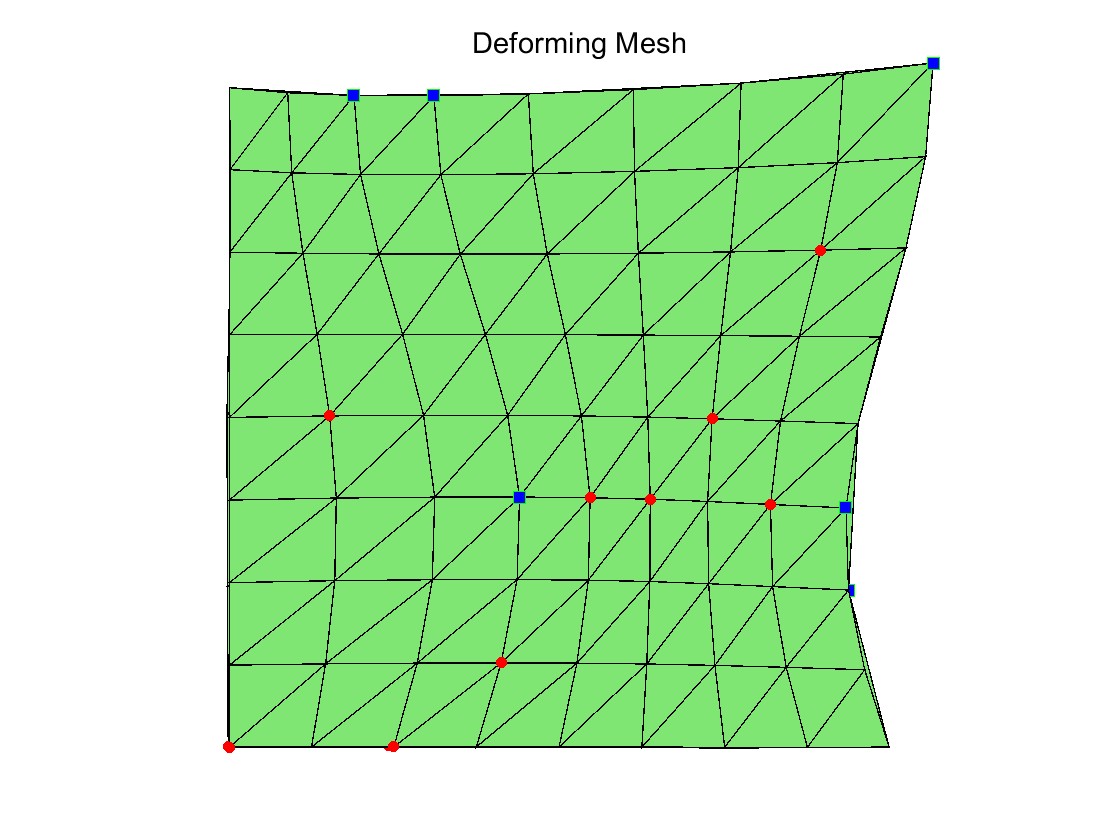}}
	\subfigure[$t=0.3$]{
		\includegraphics[width=0.45\textwidth]{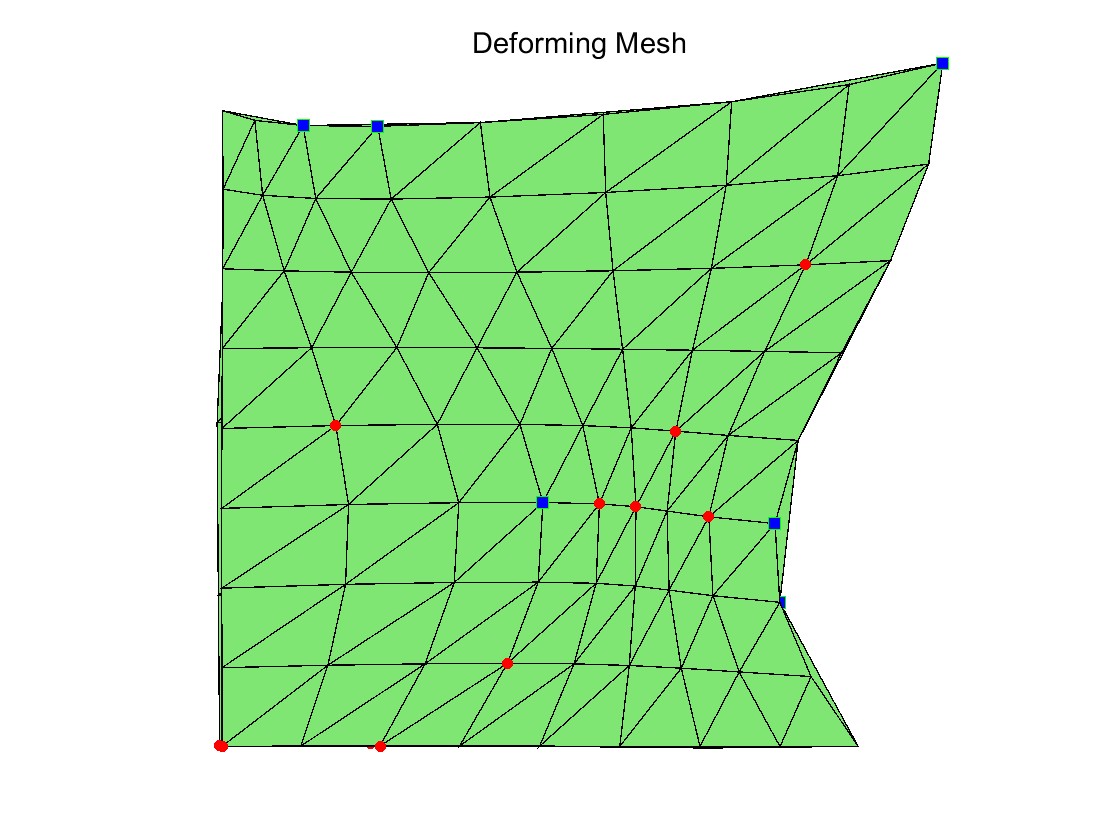}}
	\subfigure[$t=0.5$]{
		\includegraphics[width=0.45\textwidth]{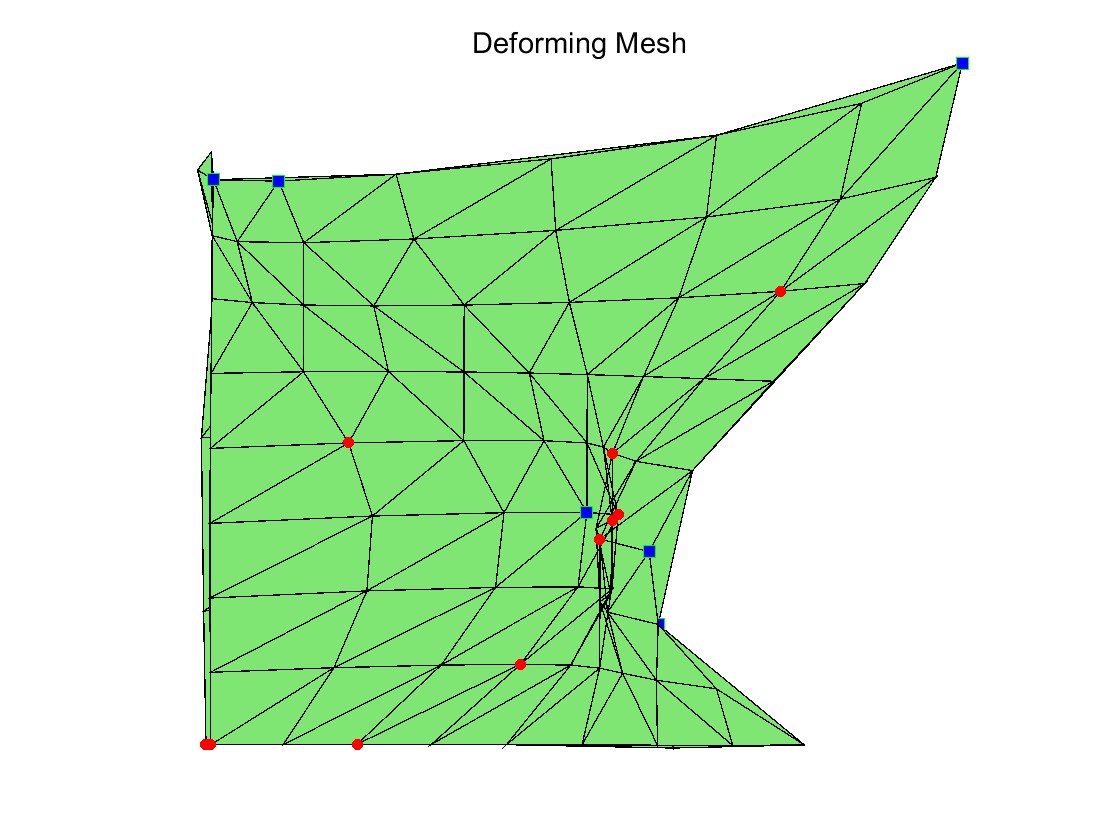}}
	\caption{The deformation of mesh over time $t$. Red dots mark nodes for deformation comparison, while blue dots highlight nodes with maximum displacement to emphasize the most deformed regions.}\label{figure_test3_1}
\end{figure}
\begin{figure}[H]
	\centering
	\subfigure[]{
		\includegraphics[width=0.45\textwidth]{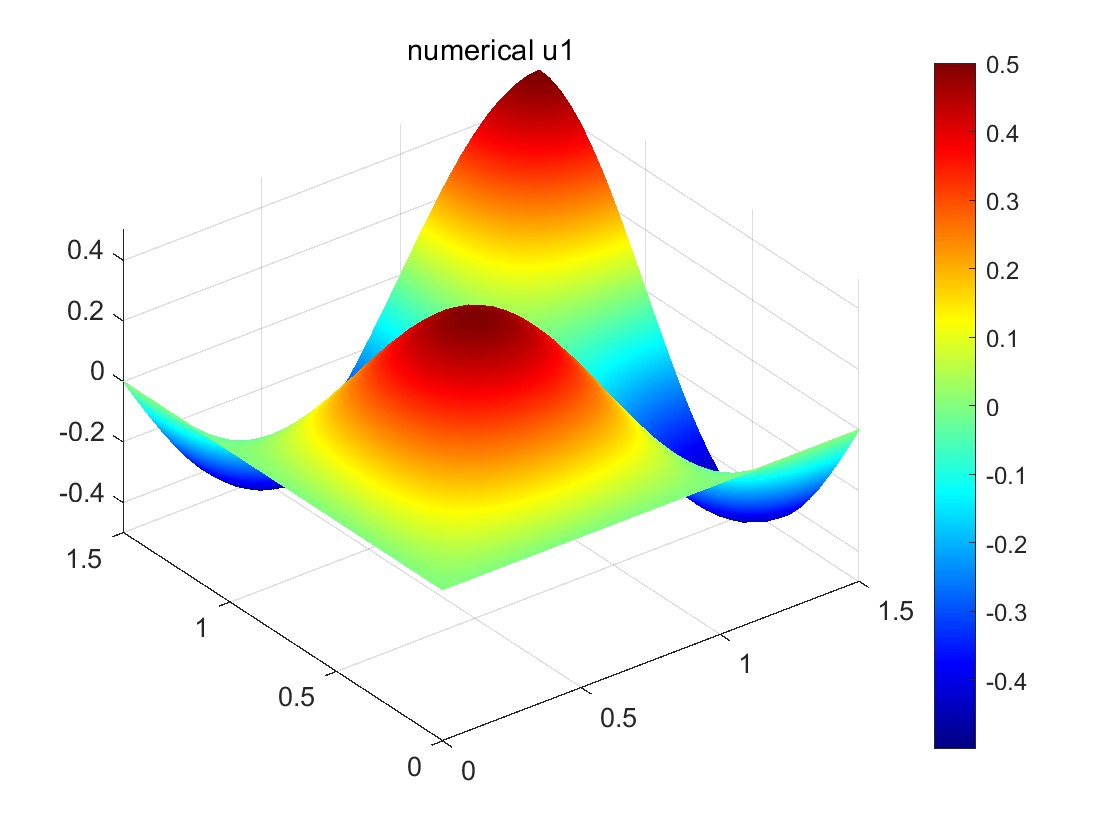}}
	\subfigure[]{
		\includegraphics[width=0.45\textwidth]{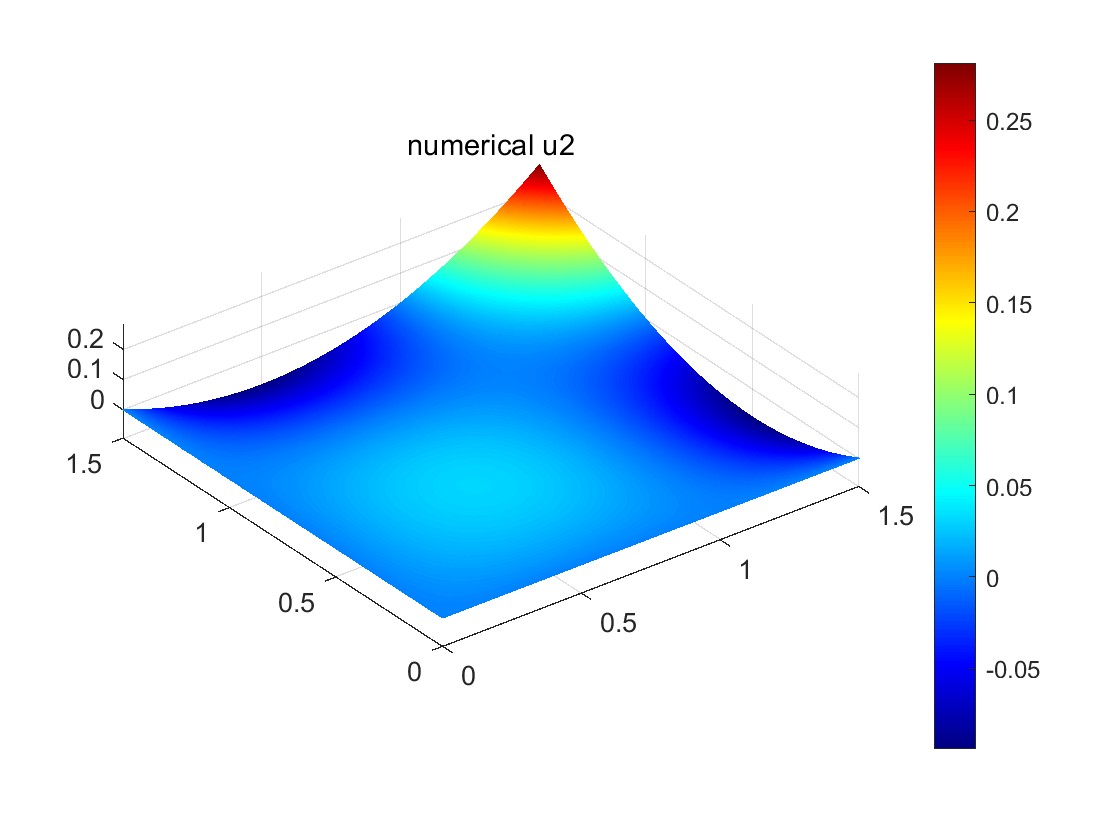}}
	\caption{Surface plots of the computed displacement (a) $u_1$   and (b) $u_2$  at the terminal time $T$.}\label{figure_test3_2}
\end{figure}

Table \ref{table_test3_2} displays the computed $L^2( \Omega) $-norm and $H^1( \Omega)$-norm errors of $\mathbf{u}$, $p$ and the convergence orders with respect to $h$ at the terminal time $T$.
Evidently, the convergence orders are consistent with Theorem \ref{thm3.5} and Theorem \ref{thm2.5}. 
Figure \ref{figure_test3_1} displays mesh deformation over time $t$. Red dots mark nodes to track position changes, while blue dots indicate nodes with the largest displacement, highlighting areas of maximum deformation. This visual aid helps compare local and global deformation patterns.
Figure \ref{figure_test3_2} displays surface plots of the computed displacements $u_1$ and $u_2$ at terminal time $T$ for mesh size $h= \frac{1}{16}$. The numerical solutions closely match the exact solution.
\begin{example}
Brain edema simulation.
\end{example}
In this numerical test, we use Algorithm \ref{alg2.1} with $\theta=1$ and ignore the influence of gravity and the body force. 
The geometric model is a 2D cross-section of a 3D model (see Figure \ref{f5.1}), one can refer to \cite{B13},  the length and width are $124$ mm and $104$ mm, respectively, $\Gamma_{2} $ is the ventricular wall with cerebrospinal fluid (CSF) inside, $\Gamma_{1} $ is the brain tissue wall whose outer part is the subarachnoid space (SAS) part.
\begin{figure}[H]
	\centering{
		\includegraphics[width=6cm]{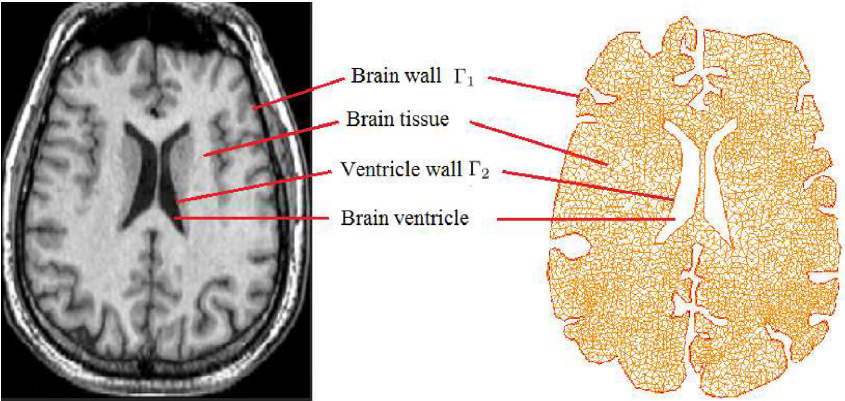}}
	\caption{Magnetic resonance imaging slice of human brain (left) and the finite element mesh (right)}\label{f5.1}
\end{figure}

The suitable boundary conditions are described below.
$\Gamma_{1}$is the brain tissue wall which is closed to the skull, so the displacement along $\Gamma_{1}$ is zero, i.e.,
\begin{equation}
	\mathbf{u}=\mathbf{0} ~~\mathrm{on}~~\Gamma_{1}.
\end{equation}

When CSF flowing out of the brain tissue, it is absorbed by the SAS part. The CSF absorption is linearly dependent on the difference value of the pressure on the brain tissue wall and the pressure of SAS ($p_{SAS}$). The balance of flow rate leads to
\begin{equation}
	(K \nabla p) \cdot \boldsymbol{n}=c_{b}\left(p_{S A S}-p\right) \quad ~~\mathrm{on}~~\Gamma_{1},
\end{equation}
where $c_{b}$ is the value of conductance. According to \cite{B36,B37,B38}, the ventricular CSF flows out of the ventricle from the aqueduct satisfies Darcy's law. From the data provided in \cite{B38}, a normal brain will produce (discharge) $0.38$ ml/min CSF, and the rate of CSF outflowing from the aqueduct is approximately $0.31$ ml/min. This means that the rate of CSF outflows through brain parenchyma is $Q_{0}=0.07$ ml/min. The conductance $c_{b}$ can be calculated by
\begin{equation}
	c_{b}=\frac{Q_{0}}{p_{d} A_{S A S}},
\end{equation}
here, $p_d= 30$ Pa is the difference between the ventricular pressure ($\thickapprox 1100$ Pa) and $p_{SAS}(\thickapprox 1070 \mathrm{Pa})$ for a normal person; 
$A_{S A S}$ is the surface area of the SAS, approximately equals $76000 ~\mathrm{mm^{2}}$, which is $\frac{1}{3}$ of the area of the cerebral cortex(see \cite{B39}). Therefore, we have $c_{b}=3 \times 10^{-5}~\mathrm{mm/min/Pa}$.

On the ventricle wall $\Gamma_{2}$, the total normal force from the tissue part needs to be balanced with the fluid pressure from the ventricle, that is:
\begin{equation}
	(\sigma-\alpha p) \cdot \boldsymbol{n}=-p \cdot \boldsymbol{n} \quad ~~\mathrm{on}~~\Gamma_{2}.
\end{equation}

The pressure at the ventricle wall is around $p=1100 ~\mathrm{Pa} ~~\mathrm{on}~~\Gamma_{2}$, one can see \cite{B31}.

\begin{table}[H]
	\begin{center}
		\caption{Parameter value}\label{t5.1}
		\begin{tabular}{cccc}
			\hline Parameters & Values & Parameters & Values \\
			\hline$c_{0}$ & $4.5 \times 10^{-7} \mathrm{~Pa}^{-1}$ & $K$ & $1.4 \times 10^{-9} \mathrm{~mm}^{2}$ \\
			$c_{b}$ & $3 \times 10^{-5} \mathrm{~mm} / \mathrm{min} / \mathrm{Pa}$ & $\alpha$ & 1 \\
			$p_{S A S}$ & $1070 \mathrm{~Pa}$ & $\mu$ & $0.35$\\
			$\mu_{f}$ & $1.48 \times 10^{-5} \mathrm{~Pa} \cdot \mathrm{min}$ & $E$ & $9010 \mathrm{~Pa}$ \\
			\hline
		\end{tabular}
	\end{center}
\end{table}

When the brain is normal, the absorption and drainage of cerebrospinal fluid are in balance, that is, $\phi=0$. 
Ref. \cite{B13} presents that the pressure distribution of a normal brain lies between $1070$ Pa– $1100$ Pa. 
In this paper, using the parameter values in Table \ref{t5.1} and Algorithm \ref{alg2.1}, we obtain that the pressure distribution is between $1067$ Pa and $1108$ Pa, which is similar to ones of \cite{B13}.

Following traumatic brain injury (TBI), the physiological equilibrium between cerebrospinal fluid absorption and drainage becomes readily disrupted. Pathological absorption at the lesion site induces cerebral tissue deformation with subsequent elevation of local intracranial pressure. Concurrently, the rigid cranial confinement creates mechanical compression of ventricular structures by displaced brain matter. 
According to the experimental data in \cite{B31,B13},
we take $\phi = 9\times10^{- 3}/\mathrm{min} $.
We take the model as the baseline model which with the physical parameters in Table \ref {t5.1} and the above $\phi$. 
\begin{figure}[H]
	\centering
	\subfigure[Pressure distribution ]{
		\includegraphics[width=0.45\textwidth]{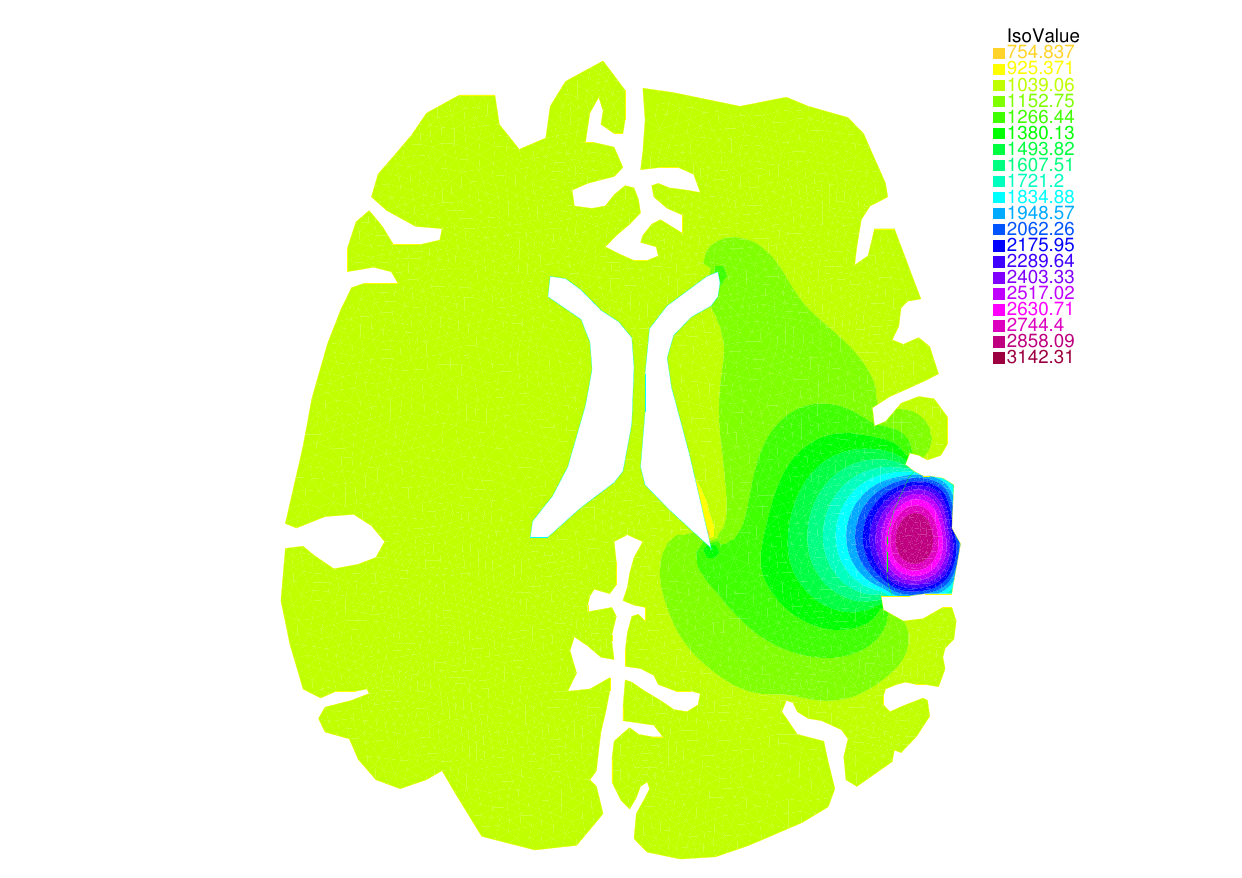}}
	\subfigure[Displacement distribution ]{
		\includegraphics[width=0.45\textwidth]{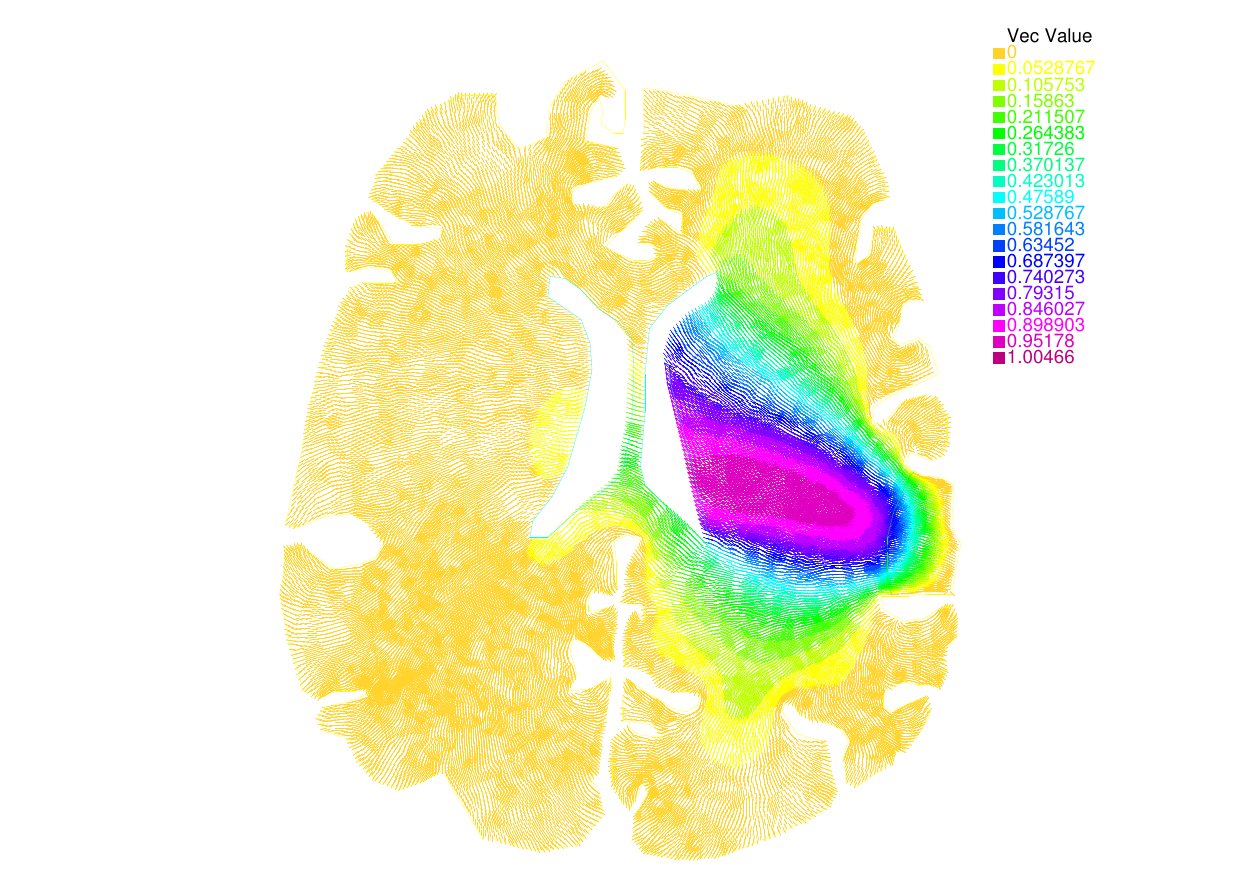}}
	\caption{
		Distribution of pressure and displacement after traumatic brain injury.}\label{f5.4}
\end{figure}
\begin{figure}[H]
	\centering{
		\includegraphics[width=8cm]{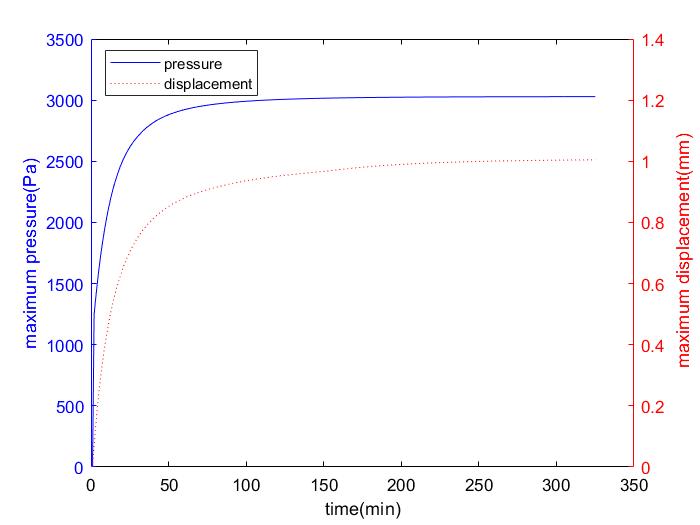}}
	\caption{Distribution of maximum displacement and pressure with time.}\label{f5.5}
\end{figure}
Figure \ref{f5.4} shows the distribution of intracranial pressure and brain tissue displacement of the injured brain. The maximum pressure $p_ {max} $ in injured area is $3142.31$ Pa. Influenced by the total stress, the brain tissue in the edema area deforms and compresses the surrounding brain tissue. Because the skull is fixed and the ventricle is free, brain tissue deformation moves toward the ventricle. In our simulation, the maximum tissue deformation $\mathbf{u}_{max} $ is $1.0046$ mm. 
This is in line with the results of \cite{B13}.

Figure \ref{f5.5} displays the maximum values of pressure and tissue deformation as functions of time, which show that the maximum intracranial pressure and tissue deformation increased rapidly in the first hour. Then, the increasing speed will slow down. Intracranial pressure and tissue deformation peaked at about $5.3$ hours.

Next, we investigate the influence of key physical parameters $E, \nu ~\mathrm{and}~ K$,  which is the main different with \cite{2022gehe}. To do that, we take the other parameters same as those in Table \ref{t5.1}. Set $	E_ {0}=9010\mathrm{~Pa}, \nu_ {0}=0.35,  K_ {0}=1.4\times10^{-9} \mathrm{~mm^{2}}$,  
 $p_ {max} = 3142.31$ Pa and $\mathbf{u}_{max} = 1.0046$ mm as the baseline values.
\begin{table}[h]
	\begin{center}
		\caption{The maximum values of $\mathbf{u}$ and $p$ $(\mathbf{u}_{max} \mathrm{~and~} p_{max})$ under different values of $E$ by fixing $\nu=\nu_{0}$ and $K=K_{0}$.}\label{t5.2}
		\begin{tabular}{ccccc}
			\hline  & $\mu$ & $1/\lambda$ & $\mathbf{u}_{max}$ & $p_{max}$ \\
			\hline  $E=0.2E_{0}$& $667\mathrm{~Pa}$ & $6.42\times 10^{-4} \mathrm{~Pa}^{-1}$ & $4.975 \mathrm{~mm}$ & $3025.97\mathrm{~Pa}$ \\
			\hline $E=2E_{0}$ & $6674\mathrm{~Pa}$ & $6.42\times 10^{-5} \mathrm{~Pa}^{-1}$ & $0.504 \mathrm{~mm}$ & $3029.36\mathrm{~Pa}$ \\
			\hline
		\end{tabular}
	\end{center}
\end{table}

\begin{figure}[H]
	\centering
	\subfigure[$E=0.2E_{0}$]{
		\includegraphics[width=0.45\textwidth]{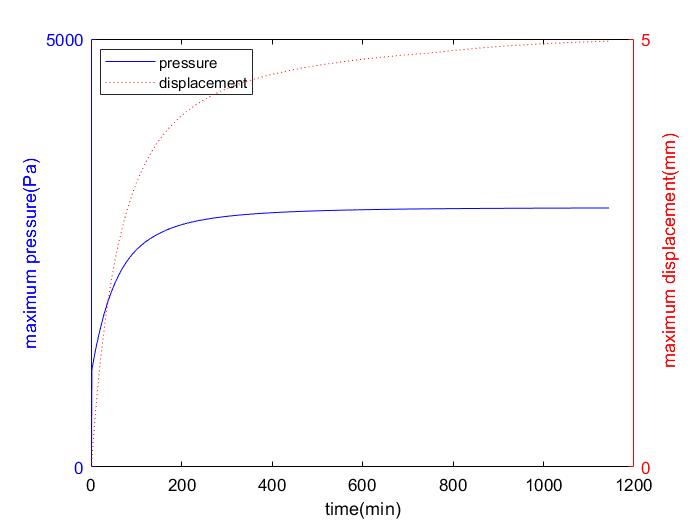}}
	\subfigure[$E=2E_{0}$]{
		\includegraphics[width=0.45\textwidth]{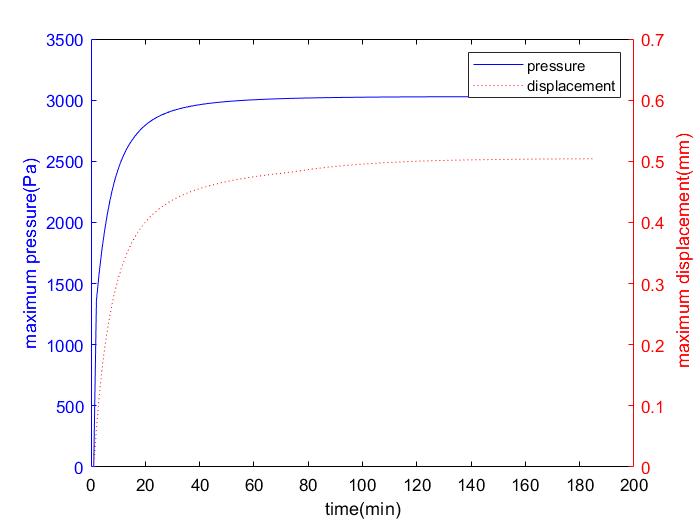}}
	\caption{The maximum values of displacement and pressure as time evolves.}\label{f5.6}
\end{figure}

Table \ref {t5.2} and Figure \ref{f5.6} display the effects of Young's modulus $E$ on the values of $\mathbf{u}_{max} $ and  $p_{max}$.  From Table \ref{t5.2}, one can observe that $\mathbf{u}_{max} $ is $4.975$ mm and $0.504$ mm which are $4.97$ and $0.5$ times of the baseline value when $E$ is $0.2E_ {0} $ and $2E_{0}$, respectively, and the change of $E$ has small effects on the pressure value. Figure \ref {f5.6} illustrates that when $E$ ranges from $0.2E_ {0} $ to $2E_{0} $, the total development time are $1144$ and $184$ minutes, respectively, which are $3.54$ and $0.57$ times of that for the baseline model. This means that it has a big influence on the edema speed. Larger $E$ will make brain edema develop much faster than expected.
\begin{table}[H]
	\begin{center}
		\caption{The maximum values of $\mathbf{u}$ and $p$ $(\mathbf{u}_{max} \mathrm{~and~} p_{max})$ under different values of $\nu$ by fixing $E=E_{0}$ and $K=K_{0}$.}\label{t5.3}
		\begin{tabular}{ccccc}
			\hline  & $\mu$ & $1/\lambda$ & $\mathbf{u}_{max}$ & $p_{max}$ \\
			\hline  $\nu=0.3$& $3465\mathrm{~Pa}$ & $1.9\times 10^{-4} \mathrm{~Pa}^{-1}$ & $1.165 \mathrm{~mm}$ & $3028.71\mathrm{~Pa}$ \\
			\hline $\nu=0.496$ & $3011\mathrm{~Pa}$ & $2.68\times 10^{-6} \mathrm{~Pa}^{-1}$ & $0.054 \mathrm{~mm}$ & $3039.05\mathrm{~Pa}$ \\
			\hline
		\end{tabular}
	\end{center}
\end{table}
\begin{figure}[H]
	\centering
	\subfigure[$\nu=0.3$]{
		\includegraphics[width=0.45\textwidth]{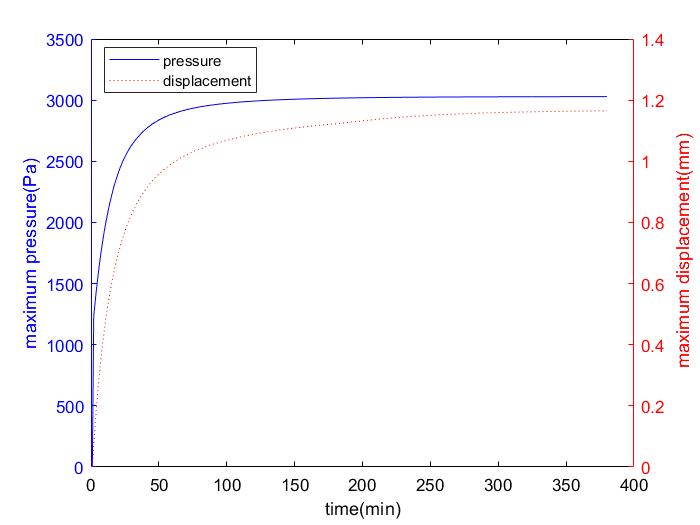}}
	\subfigure[$\nu=0.496$]{
		\includegraphics[width=0.45\textwidth]{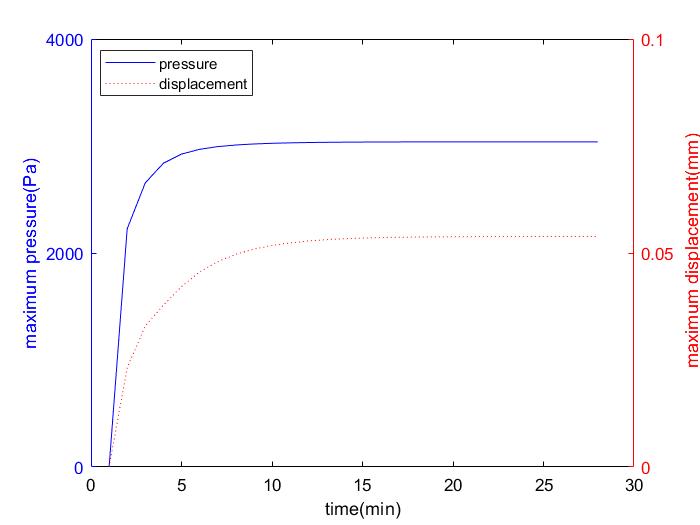}}
	\caption{The maximum values of displacement and pressure as time evolves.}\label{f5.7}
\end{figure}
Table \ref{t5.3} and Figure \ref{f5.7} display the effects of the Poisson ratio $\nu $ on the values of $\mathbf{u}_{max} $ and $p_{max} $.  Table \ref{t5.3} show that the corresponding $\mathbf {u}_{max} $ values are $1.16$ and $0.05$ times of that in baseline model when $\nu $ ranges from $0.3$ to $0.496$. However, one can observe a very small $\mathbf{u}_{max} $ when $\nu $ is close to $0.5$, which means that brain tissue is nearly incompressible.
\begin{table}[H]
	\begin{center}
		\caption{The maximum values of $\mathbf{u}$ and $p$ $(\mathbf{u}_{max} \mathrm{~and~} p_{max})$ under different values of $K$ by fixing $E=E_{0}$ and $\nu=\nu_{0}$.}\label{t5.4}
		\begin{tabular}{ccccc}
			\hline  & $\mu$ & $1/\lambda$ & $\mathbf{u}_{max}$ & $p_{max}$ \\
			\hline  $K=0.1K_{0}$& $3337\mathrm{~Pa}$ & $1.28\times 10^{-4} \mathrm{~Pa}^{-1}$ & $5.299 \mathrm{~mm}$ & $13815.8\mathrm{~Pa}$ \\
			\hline $K=10K_{0}$ & $3337\mathrm{~Pa}$ & $1.28\times 10^{-4} \mathrm{~Pa}^{-1}$ & $0.355 \mathrm{~mm}$ & $1625.35\mathrm{~Pa}$ \\
			\hline
		\end{tabular}
	\end{center}
\end{table}

\begin{figure}[H]
	\centering
	\subfigure[$K=0.1K_{0}$]{
		\includegraphics[width=0.45\textwidth]{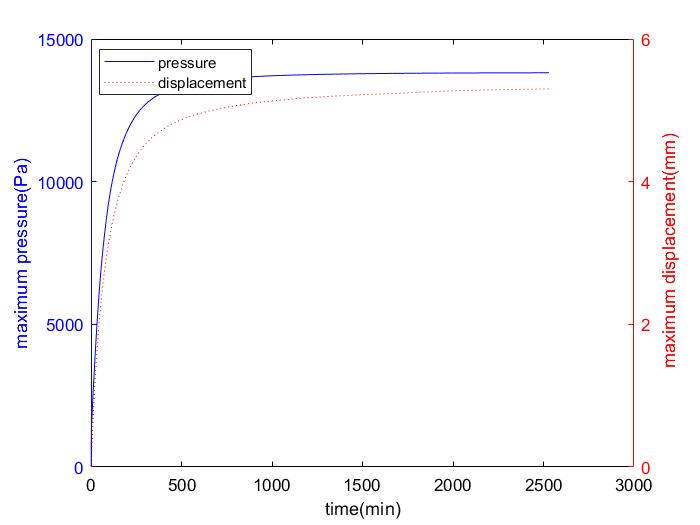}}
	\subfigure[$K=10K_{0}$]{
		\includegraphics[width=0.45\textwidth]{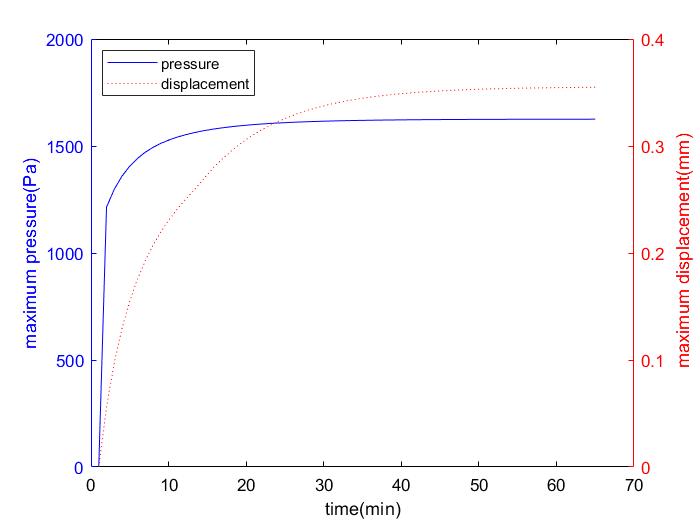}}
	\caption{The maximum values of displacement and pressure as time evolves.}\label{f5.8}
\end{figure}
Table \ref{t5.4} and Figure \ref{f5.8} display the effects of permeability $K$ on the values of $\mathbf{u}_{max} $ and $p_{max}$.
Table \ref{t5.4} shows that $\mathbf {u}_{max} $ is $530\%$ (or $35\%$) of that in the baseline value while $p_{max}$ is $456\%$ (or $54\%$) times of the baseline value when the testing permeability $K$ is $0.1K_{0}$ (or $10K_{0} $).  Meanwhile, the 
development time of  the two cases are $42$ hours and $1$ hour which are $7.83$ and $0.21$ times that of the baseline model, respectively. Unlike Poisson ratio $\nu $ and Young's modulus $E$, which only affect the tissue deformation, the permeability $K$ has a big influence on $\mathbf{u}_{max} $ and $p_{max}$. The lower permeability will result in higher pressure and larger  deformation, and therefore will make brain edema more severe.  However, the higher permeability will make brain edema develop much faster. In a word, we observe that Young's modulus $E$ and Poisson ratio $\nu$ have big influence on the value of $\mathbf{u}_{max}$ but small influence on the value of $p_{max}$, while the permeability $K$ has big influences on both of them. Moreover, we conclude that larger values of the parameters $E$, $\nu $ and $K$ will result in a smaller total developing time.

\section{Conclusion}\label{sec-4}
In this paper, we derive an optimal  $L^2$-norm error estimate of the  multiphysics finite element method for the poroelasticity model  by introducing an auxiliary problem. Then, we show some numerical tests to verify the theoretical result and apply the reformulated model to simulate brain edema and investigate the effects  of the key parameters and compare the results with the existing work, which will be helpful to model and cure the brain edema.



\begin{thebibliography}{99}

\bibitem{B1}
R. Lewis, B. Schrefler. The Finite Element Method in the Static and Dynamic Deformation and Consolidation in Porous Media. John Wiley and Sons, 1998.


\bibitem{B2}
A. Settari, D. Walters. Advances in coupled geomechanical and reservoir modeling with applications to reservoir compaction. SPE Journal, 2001, 6(3):334–342.


\bibitem{B3}
A. Settari, D. Walters, G. Behie. Use of coupled reservoir and geomechanical modeling for integrated reservoir analysis and management. Journal of Canadian Petroleum Technology, 2001, 40(12): PETSOC-01-12-04.


\bibitem{B4}
M. Biot. General theory of three-dimensional consolidation. Journal of Applied Physics, 1941, 12: 155-164.


\bibitem{B5}
G. Chen. Consolidation of multilayered half space with anisotropic permeability and compressible constituents. International Journal of Solids and Structures, 2004, 41: 4567-4586.


\bibitem{B6}
Y. Cheung, L. Tham. Numerical solutions for Biot's consolidation of layered soil. Journal of Engineering Mechanics, 1983, 109(3): 669-679.


\bibitem{B7}
J. Hudson, O. Stephansson, J. Andersson, C. Tsang, L. Jing. Coupled T-H-M issues relating to radioactive waste repository design and performance. International Journal of Rock
Mechanics and Mining Sciences, 2001, 38(1): 143-161.


\bibitem{B8}
M. Ferronato, N. Castelletto, G. Gambolati. A fully coupled 3-D mixed finite element model of Biot consolidation. Journal of Computational Physics, 2010, 229(12): 4813-4830.



\bibitem{B9}
B. Simon, M. Kaufmann, M. McAfee, A. Baldwin. Finite element models for arterial wall mechanics. Journal of Biomechanical Engineering, 1993, 115: 489-496.


\bibitem{B10}
A. Mak, L. Huang, Q. Wang. A biphasic poroelastic analysis of the flow dependent subcutaneous tissue pressure and compaction due to epidermal loadings: issues in pressure sore. Journal of Biomechanical Engineering, 1994, 116(4): 421-429.

\bibitem{B11}
M. Yang, L. Taber. The possible role of poroelasticity in the apparent viscoelastic behavior of passive cardiac muscle. Journal of Biomechanics, 1991, 24(7): 587-597.


\bibitem{B12}
V. Mow, M. Kwan, W. Lai, M. Holmes, A finite deformation theory for nonlinearly permeable soft hydrated biological tissues. Springer-Verlag New York Inc, 1986, pp 153–179.


\bibitem{B13}
G. Ju, M. Cai, J. Li, J. Tian. Parameter-robust multiphysics algorithms for Biot model with application in brain edema simulation. Mathematics and Computers in Simulation, 2020, 177: 385-403.


\bibitem{B14}
P. Phillips, M. Wheeler. A coupling of mixed and continuous Galerkin finite element methods for poroelasticity I: the continuous in time case. Computational Geosciences, 2007, 11(2): 131-144.


\bibitem{B15}
P. Phillips, M. Wheeler. A coupling of mixed and continuous Galerkin finite element methods for poroelasticity II: the discrete in time case. Computational Geosciences, 2007, 11(2): 145-158.


\bibitem{B41}
P. Phillips, M. Wheeler. A coupling of mixed and discontinuous Galerkin finite-element methods for poroelasticity. Computational Geosciences, 2008, 12(4): 417-435.


\bibitem{Phillips2009}
P. Phillips, M. Wheeler. Overcoming the problem of locking in linear elasticity and poroelasticity: an heuristic approach. Computational Geosciences, 2009, 13(1): 5–12. 


\bibitem{B42}
S. Yi. A coupling of nonconforming and mixed finite element methods for Biot's consolidation model. Numerical Methods for Partial Differential Equations, 2013, 29(5): 1749-1777.


\bibitem{B27} X. Feng, Z. Ge, Y. Li. Analysis of a multiphysics finite element method for a poroelasticity model. IMA Journal of Numerical Analysis, 2018, 38(1): 330-359.  arXiv:1411.7464, [math.NA], 2014.


\bibitem{B43}
M. Sun, H. Rui. A coupling of weak Galerkin and mixed finite element methods for poroelasticity. Computers And Mathematics with Applications, 2017, 73(5): 804-823.


\bibitem{B32}
S. Budday, R. Nay, R. Rooij, P. Steinmann, T. Wyrobek, T. Ovaert, E. Kuhl, Mechanical properties of gray and white matter brain tissue by indentation. Journal of the Mechanical Behavior of Biomedical Materials, 2015, 46: 318-330.


\bibitem{B34}
S. Hakim, J. Venegas, J. Burton. The physics of the cranial cavity, hydrocephalus and normal pressure hydrocephalus: mechanical interpretation and mathematical model. Surgical Neurology International, 1976, 5(3): 187-210.



\bibitem{temam}  
R. Temam. Navier-Stokes Equations.
Studies in Mathematics and its Applications, Vol. 2, North-Holland, 1977.

\bibitem{B28}
S. Brenner, L. Scoot. The Mathematical Theory of Finite Element Methods, Third Edition. Springer-Verlag New York Inc, 2008.



\bibitem{B22}
S. Brenner. A nonconforming mixed multigrid method for the pure displacement problem in planar linear elasticity. SIAM Journal on Numerical Analysis, 1993, 30(1): 116-135.


\bibitem{B23}
V. Girault, P. Raviart. Finite Element Methods for Navier-Stokes Equations: Theory and Algorithms. Springer-Verlag Berlin Heidelberg, 1986.


\bibitem{B24}
M. Bebendorf. A note on the Poincar$\acute{e}$ inequality for convex domains. Journal for Analysis and its Applications, 2003, 22(4): 751-756.


\bibitem{B35}
F. Brezzi, M. Fortin. Mixed and Hybrid Finite Element Methods. Springer-Verlag New York Inc, 1991.


\bibitem{B26}
X. Feng, Y. He. Fully discrete finite element approximations of a polymer gel model. SIAM Journal on Numerical Analysis, 2010, 48(6): 2186-2217.



\bibitem{B38}
J. Vardakis, D. Chou, B. Tully, C. Hung, T. Lee, P. Tsui, Y. Ventikos. Investigating cerebral oedema using poroelasticity. Medical Engineering and Physics, 2016, 38(1): 48-57.



\bibitem{B36}
D. Levine. The pathogenesis of normal pressure hydrocephalus: a theoretical analysis. Bulletin of mathematical biology, 1999, 61(5): 875-916.


\bibitem{B37}
A. Smillie, I. Sobey, Z. Molnar. A hydroelastic model of hydrocephalus. Journal of Fluid Mechanics, 2005, 539: 417-443.


\bibitem{B39}
R. Toro, M. Perron, B. Pike, L. Richer, S. Veillette, Z. Pausova, T. Paus. Brain size and folding of the human cerebral cortex. Cerebral Cortex. 2008, 18(10): 2352-2357.


\bibitem{B31}
X. Li, H. Holst, S. Kleiven. Influence of gravity for optimal head positions in the treatment of head injury patients. Acta Neurochirurgica, 2011, 153(10): 2057-2064.




\bibitem{2022gehe}
W. He, Z. Ge. A new mixed finite element method for a swelling clay model with secondary consolidation. Applied Mathematical Modelling, 2020, 112: 391-414.



\end{thebibliography}
\end{document}